\pdfoutput=1
\RequirePackage{ifpdf}
\ifpdf 
\documentclass[pdftex]{sigma}
\else
\documentclass{sigma}
\fi

\numberwithin{equation}{section}

\newtheorem{Theorem}{Theorem}[section]
\newtheorem{Corollary}[Theorem]{Corollary}
\newtheorem{Lemma}[Theorem]{Lemma}
\newtheorem{Proposition}[Theorem]{Proposition}
 { \theoremstyle{definition}
\newtheorem{Remark}[Theorem]{Remark} }

\def\a{\alpha}
\def\b{\beta}
\newcommand{\rr}{\mathbb{R}}
\newcommand{\nn}{\mathbb{N}}

\def\g{\gamma}
\newcommand{\hypergeom}[5]{\mbox{$
_#1 F_#2\left(\left.
\begin{matrix}
\multicolumn{1}{c}{\begin{matrix} #3
\end{matrix}}\\[1mm]
\multicolumn{1}{c}{\begin{matrix} #4
 \end{matrix}}\end{matrix}
 \right| \displaystyle{#5}\right) $} }
\newcommand{\qhypergeom}[5]{\mbox{$
_#1 \phi_#2\left( \left.
\begin{matrix}
\multicolumn{1}{c}{\begin{matrix} #3
\end{matrix}}\\[1mm]
\multicolumn{1}{c}{\begin{matrix} #4
 \end{matrix}}\end{matrix}
 \right| \displaystyle{#5}\right) $} }

\begin{document}

\allowdisplaybreaks

\renewcommand{\thefootnote}{}

\renewcommand{\PaperNumber}{051}

\FirstPageHeading

\ShortArticleName{Quasi-Orthogonality of Some Hypergeometric and $q$-Hypergeometric Polynomials}

\ArticleName{Quasi-Orthogonality of Some Hypergeometric\\ and $\boldsymbol{q}$-Hypergeometric Polynomials\footnote{This paper is a~contribution to the Special Issue on Orthogonal Polynomials, Special Functions and Applications (OPSFA14). The full collection is available at \href{https://www.emis.de/journals/SIGMA/OPSFA2017.html}{https://www.emis.de/journals/SIGMA/OPSFA2017.html}}}

\Author{Daniel D.~TCHEUTIA~$^\dag$, Alta S.~JOOSTE~$^\ddag$ and Wolfram KOEPF~$^\dag$}

\AuthorNameForHeading{D.D.~Tcheutia, A.S.~Jooste and W.~Koepf}

\Address{$^\dag$~Institute of Mathematics, University of Kassel,\\
\hphantom{$^\dag$}~Heinrich-Plett Str.~40, 34132 Kassel, Germany}

\EmailD{\href{mailto:duvtcheutia@yahoo.fr}{duvtcheutia@yahoo.fr}, \href{mailto:koepf@mathematik.uni-kassel.de}{koepf@mathematik.uni-kassel.de}}
\URLaddressD{\url{http://www.mathematik.uni-kassel.de/~koepf/}}

\Address{$^\ddag$~Department of Mathematics and Applied Mathematics, University of Pretoria,\\
\hphantom{$^\ddag$}~Pretoria 0002, South Africa}
\EmailD{\href{mailto:alta.jooste@up.ac.za}{alta.jooste@up.ac.za}}

\ArticleDates{Received January 26, 2018, in final form May 17, 2018; Published online May 23, 2018}

\Abstract{We show how to obtain linear combinations of polynomials in an orthogonal sequence $\{P_n\}_{n\geq 0}$, such as $Q_{n,k}(x)=\sum\limits_{i=0}^k a_{n,i}P_{n-i}(x)$, $a_{n,0}a_{n,k}\neq0$, that characterize quasi-orthogonal polynomials of order $k\le n-1$. The polynomials in the sequence $\{Q_{n,k}\}_{n\geq 0}$ are obtained from $P_{n}$, by making use of parameter shifts. We use an algorithmic approach to find these linear combinations for each family applicable and these equations are used to prove quasi-orthogonality of order~$k$. We also determine the location of the extreme zeros of the quasi-orthogonal polynomials with respect to the end points of the interval of orthogonality of the sequence $\{P_n\}_{n\geq 0}$, where possible.}


\Classification{33C05; 33C45; 33F10; 33D15; 12D10}

\renewcommand{\thefootnote}{\arabic{footnote}}
\setcounter{footnote}{0}

\section{Introduction}
A sequence of polynomials $\{P_n\}_{n\ge0}$, where each polynomial $P_n$ {has} degree $n$, is orthogonal with respect to the weight function $w(x)>0$ on the finite (or infinite) interval $[a,b]$ if
\begin{gather*}
 \int_{a}^{b}x^mP_{n}(x)w(x){\rm d}x=0, \qquad m\in\{0,1,\ldots,n-1\}.
 \end{gather*}
A consequence of orthogonality is that each polynomial $P_n(x)$ has $n$ real, {simple} zeros in $(a,b)$. In order for orthogonality conditions to hold, we often need restrictions on the parameters of the classical orthogonal polynomials and when the parameters deviate from these restricted values in an orderly way, the zeros may depart from the interval of orthogonality in a predictable way. This phenomenon can be explained in terms of the concept of quasi-orthogonality. The sequence of polynomials $\{Q_{n,k}\}_{n\ge0}$, where each polynomial $Q_{n,k}$ {has} degree $n$, is quasi-orthogonal of order $k\in\{1,2,\dots,n-1\}$ with respect to the weight function $w(x)>0$ on $[a,b]$ if
 \begin{gather}\label{quasi_orth}
 \int_{a}^{b}x^mQ_{n,k}(x)w(x){\rm d}x=0, \qquad m\in\{0,1,\ldots,n-k-1\}.
 \end{gather}

It is clear that when $k=0$ in \eqref{quasi_orth}, the sequence $\{Q_{n,k}\}_{n\geq 0}$ is orthogonal with respect to~$w(x)$ on~$[a,b]$.

Quasi-orthogonality was first studied by Riesz \cite{riesz}, followed by Fej\'{e}r \cite{fejer}, Shohat \cite{Shohat}, Chihara \cite{Chihara}, Dickinson \cite{Dickinson}, Draux \cite{Draux_1990}, Maroni \cite{maroni} and Joulak \cite{Joulak_2005}. The quasi-orthogonality of Jacobi, Gegenbauer and Laguerre sequences is discussed in \cite{BDR}, the quasi-orthogonality of Meixner sequences in \cite{JJT} and of Meixner--Pollaczek, Hahn, dual Hahn and continuous dual Hahn sequences in \cite{JJJ}. More recently, interlacing of zeros of quasi-orthogonal Meixner, Jacobi, Laguerre and Gegenbauer polynomials were studied in \cite{DriverJooste, DriverJordaan,DriverMuldoonLag, DriverMuldoonGeg} and in \cite{bultheel} interlacing properties of zeros of quasi-orthogonal polynomials were used to prove results on Gaussian-type quadrature.

Quasi-orthogonal polynomials are characterized by the following property:
\begin{Lemma}[\cite{BDR, Chihara}] \label{Brez}
 Let $\{P_n\}_{n\geq 0}$ be a family of orthogonal polynomials on $[a,b]$ with respect to the weight function $w(x)>0$. A necessary and sufficient condition for a polynomial $Q_{n,k}$ of degree $n$ to be quasi-orthogonal of order $k\le n-1$ with respect to $w$ on $[a,b]$, is that
 \begin{gather}\label{charac_quasi}
 Q_{n,k}(x)=\sum_{i=0}^k a_{n,i}P_{n-i}(x),\qquad a_{n,0}a_{n,k}\neq 0, \qquad n>k.
 \end{gather}
\end{Lemma}
We show at the end of this section how to derive equations of type \eqref{charac_quasi}, where $\{P_n\}_{n\ge0}$ is a sequence of orthogonal polynomials in the Askey or $q$-Askey scheme~\cite{KLS}. The polynomials in the Askey scheme are defined in terms of the generalized hypergeometric series
\begin{gather*}\hypergeom{p}{q}{a_1,\ldots,a_p}{b_1,\ldots, b_q}{x}=\sum_{m=0}^{\infty}\frac{(a_1)_m\cdots (a_p)_m}{(b_1)_m\cdots (b_q)_m}\frac{x^m}{m!},\end{gather*}
and $(a)_m$ denotes the Pochhammer symbol (or shifted
factorial) defined by \begin{gather*}(a)_m=
 \begin{cases}
 1 & \text{if } m=0, \\
 a(a+1)(a+2)\cdots (a+m-1) & \text{if } m\in \mathbb{N}.
 \end{cases}
\end{gather*}
The polynomials in the $q$-Askey scheme are defined in terms of the basic hypergeometric series~$_r\phi_s$, given by
 \begin{gather*}
 \qhypergeom{r}{s}{a_1,\ldots,a_r}{b_1,\ldots,b_s}{q;z}=
 \sum_{m=0}^\infty\frac{(a_1,\ldots,a_r;q)_m}{(b_1,\ldots,b_s;q)_m}\big((-1)^m
 q^{\binom{m}{2}}\big)^{1+s-r}\frac{z^m}{(q;q)_m},
 \end{gather*}
 where the $q$-Pochhammer symbol $(a_1,a_2,\ldots,a_r;q)_n$ is defined by \begin{gather*}(a_1,\ldots,a_r;q)_m:=(a_1;q)_m \cdots (a_r;q)_m,
 \end{gather*}
 with
 \begin{gather*}
 (a_i;q)_m=
 \begin{cases}
 \prod\limits_{j=0}^{m-1}\big(1-a_iq^j\big)& \text{if} \ m\in\{1,2,3,\ldots\}, \\
 1 & \text{if} \ m=0.
 \end{cases}
 \end{gather*}

We will discuss the quasi-orthogonality of monic polynomial systems on a $q$-linear lattice in Section~\ref{section2}, the quasi-orthogonality of the monic Askey--Wilson and $q$-Racah polynomials, that lie on a $q$-quadratic lattice, in Section~\ref{section3}, the quasi-orthogonality of the monic Wilson, Racah and continuous Hahn polynomials in Section~\ref{section4} and in Section~\ref{section5} we will make some concluding remarks. Throughout this paper, $x_{n,j}$, $j\in\{1,2,\dots,n\}$, denote the zeros of a polynomial of degree~$n$ in the following order: $x_{n,1}<x_{n,2}<\cdots<x_{n,n}$. The following results will be referred to in this paper and we state them here for the convenience of the reader.

\begin{Lemma}[\cite{BDR, Shohat}]\label{location}
If $Q_{n,k}$ is quasi-orthogonal of order $k\geq 1$ on $[a,b]$ with respect to $w(x)>0$, then at least $(n-k)$ real, distinct zeros of $Q_{n,k}$ lie in the interval $(a,b)$.
\end{Lemma}
\begin{Lemma}[\cite{BDR, Joulak_2005}]\label{Joulak1}
 Suppose $Q_{n,1}(x)=P_n(x)+a_{n}P_{n-1}(x)$, $a_{n}\neq 0$. Let $y_{n,j}$, $j\in\{1,2,\dots,n\}$, be the zeros of $Q_{n,1}(x)$ and let $f_n(x)=\frac{P_n(x)}{P_{n-1}(x)}$. We have
 \begin{enumerate}\itemsep=0pt
 \item[$(i)$] $y_{n,1}<a$ if and only if $-a_{n}<f_n(a)<0$;
 \item[$(ii)$] $b<y_{n,n}$ if and only if $-a_{n}>f_n(b)>0$;
 \item[$(iii)$] $Q_{n,1}$ has all its zeros in $(a,b)$ if and only if $f_n(a)<-a_{n}<f_n(b)$.
 \end{enumerate}
\end{Lemma}
\begin{Lemma}[\cite{BDR, Joulak_2005}] \label{Joulak2}
 Suppose $Q_{n,1}(x)=P_n(x)+a_{n}P_{n-1}(x)$, $a_{n}\neq 0$. Let $x_{n,j}$, $j\in\{1,2,\dots,n\}$, denote the zeros of $P_n(x)$ and $y_{n,j}$, $j\in\{1,2,\dots,n\}$, the zeros of $Q_{n,1}(x)$. Then
 \begin{enumerate}\itemsep=0pt
 \item[$(i)$] $a_n<0$ if and only if $x_{n,1}<y_{n,1}<x_{n-1,1}<x_{n,2}<y_{n,2}<\dots<x_{n-1,n-1}<x_{n,n}<y_{n,n}$;
 \item[$(ii)$] $a_n>0$ if and only if $y_{n,1}<x_{n,1}<x_{n-1,1}<y_{n,2}<x_{n,2}<\dots<x_{n-1,n-1}<y_{n,n}<x_{n,n}$.
 \end{enumerate}
\end{Lemma}

In order to find the equations of type \eqref{charac_quasi}, that are needed to prove quasi-orthogonality, we consider
the structure relation (cf.\ \cite{Foupouagnigni_et_al_2012,Koepf_Schmersau_1998,Medem_et_al_2001})
\begin{gather}\label{struc_rel}
 P_n(x)=a_nDP_{n+1}(x)+b_nDP_n(x)+c_nDP_{n-1}(x),
\end{gather}
where the constants $a_n$, $b_n$ and $c_n$ are explicitly given and $D$ is a derivative or difference operator. Most of the classical orthogonal polynomials we consider in this paper (see \cite[Chapters~9 and~14]{KLS}) satisfy
\begin{gather}\label{qdiff}
 DP_n(x)=S(n)P_{n-1,k}(x), \qquad k\in\{-1,0,1,2\},
\end{gather}
where $S(n)$ does not depend on $x$ and $P_{n-1,k}(x)$ denotes the polynomial obtained when each of the parameters on which the polynomial $P_n(x)$ depends, can be shifted by $k$ units in the case of the classical systems, or, in the case of the $q$-classical systems, when the parameters can each be multiplied by $q^k$. If we substitute \eqref{qdiff} in \eqref{struc_rel}, we obtain
 \begin{gather*}
 P_n(x)=a_nS(n+1)P_{n,k}(x)+b_nS(n)P_{n-1,k}(x)+c_nS(n-1)P_{n-2,k}(x)
 \end{gather*}
or, by making a parameter shift,
\begin{gather*}
 P_{n,-k}(x)=a_n'S'(n+1)P_{n}(x)+b_n'S'(n)P_{n-1}(x)+c_n'S'(n-1)P_{n-2}(x),
 \end{gather*}
where $a_n'$, $b_n'$, $c_n'$ and $S'(n)$ are the values of the coefficients taking into consideration the parameter shift, and we have a linear combination of polynomials in an orthogonal sequence as in~(\ref{charac_quasi}). For the so-called very-classical orthogonal polynomials, the general expression for the parameters $a_{n,i}$, $i\in\{0,1,2\}$, in
\begin{gather*}Q_{n,2}(x)=P_{n}(x)+a_{n,1}P_{n-1}(x)+a_{n,2}P_{n-2}(x)\end{gather*}
 were given in \cite[equation~(76)]{MP} in terms of the coefficients of the differential equations they satisfy.

We can also apply the operator $D$ to \eqref{struc_rel} to obtain
\begin{gather}\label{eq_dqpn}
 DP_n(x)=a_nD^2P_{n+1}(x)+b_nD^2P_n(x)+c_nD^2P_{n-1}(x).
\end{gather}
Replacing \eqref{eq_dqpn} in \eqref{struc_rel} and using \eqref{qdiff} twice, yields
\begin{gather*}
 P_n(x)=a_na_{n+1}S(n+2)S(n+1)P_{n,2k}(x)+a_n(b_n+b_{n+1})S(n+1)S(n) P_{n-1,2k}(x) \\
\hphantom{P_n(x)=}{} +\big(a_nc_{n+1}+a_{n-1}c_n+b_n^2\big)S(n)S(n-1)P_{n-2,2k}(x) \\
\hphantom{P_n(x)=}{}+ c_n(b_n+b_{n-1})S(n-1) S(n-2)P_{n-3,2k}+c_nc_{n-1}S(n-2)S(n-3)P_{n-4,2k}(x).
\end{gather*}
By making a parameter shift again, we obtain
\begin{gather*}
 P_{n,-2k}(x)=a_n'a_{n+1}'S'(n+2)S'(n+1)P_{n}(x)+a_n'(b_n'+b_{n+1}')S'(n+1)S'(n) P_{n-1}(x)\\
\hphantom{P_{n,-2k}(x)=}{} +\big(a_n'c_{n+1}'+a_{n-1}'c_n'+(b_n')^2\big)S'(n)S'(n-1)P_{n-2}(x) \\
\hphantom{P_{n,-2k}(x)=}{} +c_n'(b_n'+b_{n-1}')S'(n-1) S'(n-2)P_{n-3}+c_n'c_{n-1}'S'(n-2)S'(n-3)P_{n-4}(x).
\end{gather*}

These induction arguments show that equations of type~\eqref{charac_quasi} are structurally valid. We also refer the reader to~\cite{Foupouagnigni_et_al_2012}, where we have so-called connection formulae for the classical orthogonal polynomial of the $q$-linear lattice from which one can deduce equations of type~\eqref{charac_quasi}.

Using Zeilberger's algorithm and following the approach in \cite{Chen_et_al_2012, Koepf_2014}, we wrote, using the computer algebra system Maple, a procedure denoted \texttt{MixRec}($F,k,S(n),\alpha,\beta,m$) which finds a~recurrence equation of the form
\begin{gather*}S(n,\alpha+m,\beta+m)=\sum_{j=0}^J\sigma_jS(n-j,\alpha,\beta ),\qquad J\in\{1,2,\ldots\},\end{gather*}
where $S(n,\alpha,\beta)= \sum\limits_{k=-\infty}^\infty F$, $F$ is a function of $k$, $n$, $\alpha$ and $\beta$ (of hypergeometric type), and $m$ is an integer, and its $q$-analogue denoted \texttt{qMixRec}($F,q,k,S(n),\alpha,\beta,m$) which finds a recurrence equation of the form
\begin{gather*}S\big(n,\alpha q^m,\beta q^m\big)=\sum_{j=0}^J\sigma_jS(n-j,\alpha,\beta ),\qquad J\in\{1,2,\ldots\}.\end{gather*}

Our algorithms and all the equations used in this paper can be downloaded from \url{http://www.mathematik.uni-kassel.de/~koepf/Publikationen}. In~\cite{KoepfJT_2017}, using the same algorithmic approach, the authors provide a procedure to find mixed recurrence equations satisfied by classical $q$-orthogonal polynomials with shifted parameters. These equations were used to investigate interlacing properties of zeros of sequences of $q$-orthogonal polynomials. Using our~\texttt{MixRec} function for example, we recover (as shown at the end of our Maple file) the equations used in \cite[Section~3]{BDR} to study the quasi-orthogonality of the Gegenbauer, the Laguerre, the Jacobi polynomials.

\section[Classical orthogonal polynomials on a $q$-linear lattice]{Classical orthogonal polynomials on a $\boldsymbol{q}$-linear lattice}\label{section2}

In this section we consider the quasi-orthogonality of sequences of monic orthogonal polynomials that are defined on a $q$-linear lattice, as well as the location of the zeros of the quasi-orthogonal sequences.

\subsection[The big $q$-Jacobi polynomials]{The big $\boldsymbol{q}$-Jacobi polynomials}
Big $q$-Jacobi polynomials
 \begin{gather} \label{eq:BigqJacobi}
 \tilde{P}_n(x;\alpha,\beta,\gamma;q)=\frac{(\alpha q;q)_n(\gamma q;q)_n}{(\alpha\beta q q^n;q)_n}\,\qhypergeom{3}{2}{q^{-n},\alpha\beta q^{n+1},x}{\alpha q,\gamma q}{q;q}
 \end{gather}
are orthogonal for $0<\alpha q<1$, $0\leq \beta q<1$ and $\gamma<0$ with respect to the continuous weight function $w(x)=\frac{(\a^{-1}x,\g^{-1}x;q)_{\infty}}{(x,\b\g^{-1}x;q)_{\infty}}$ on $(\g q,\a q)$.

The first two recurrence equations in the following proposition follow from \cite[equations~(7a) and~(7b)]{KoepfJT_2017}, with $\a$ and $\b$ replaced by $\frac{\a}{q}$ and $\frac{\b}{q}$, respectively. The big $q$-Jacobi polynomials are orthogonal for $\g<0$, and by replacing $\g$ by $\frac{\g}{q}$, $0<q<1$, we obtain the polynomial $\tilde{P}_n\big(x;\alpha,\beta,\frac{\gamma}{q};q\big)$ of which all the parameters are still in the regions where orthogonality is guaranteed and we will therefore not consider a~$q$-shift of~$\gamma$.
\begin{Proposition}
\begin{subequations}
\begin{gather}
\tilde{P}_n\left(x;\frac{\alpha}{q},\beta,\gamma;q\right) =\tilde{P}_n(x;\alpha,\beta,\gamma;q)\nonumber\\
\hphantom{\tilde{P}_n\left(x;\frac{\alpha}{q},\beta,\gamma;q\right) =}{} +\frac { \alpha q\left( {q}^{n}-1 \right) \left( \beta {q}^{n}-1 \right) \left( \gamma {q}^{n}-1 \right) }{ \left( \alpha \beta {q}^{2 n}-1 \right) \left( \alpha
 \beta {q}^{2 n}-q \right) }\tilde{P}_{n-1}(x;\alpha,\beta,\gamma;q);\label{eq1bqj}\\
 \tilde{P}_n\left(x;\alpha,\frac{\beta}{q},\gamma;q\right) =\tilde{P}_n(x;\alpha,\beta,\gamma;q) \nonumber\\
 \hphantom{\tilde{P}_n\left(x;\alpha,\frac{\beta}{q},\gamma;q\right) =}{} -\frac { \alpha\beta{q}^{n+1}\left( {q}^{n}-1 \right) \left( \alpha q^n-1 \right)\left( \gamma q^n-1 \right) }{ \left( \alpha \beta {q}^{2 n}-1 \right)
 \left( \alpha \beta {q}^{2 n}-q \right) }\tilde{P}_{n-1}(x;\alpha,\beta,\gamma;q);\label{eq3bqj}\\
 \tilde{P}_n\left(x;\alpha,\frac{\beta}{q},\frac{\gamma}{q};q\right) =\tilde{P}_n(x;\alpha,\beta,\gamma;q)\nonumber\\
 \hphantom{\tilde{P}_n\left(x;\alpha,\frac{\beta}{q},\frac{\gamma}{q};q\right) =}{} -\frac { \left( {q}^{n}-1 \right) \left( \alpha {q}^{n}-1 \right) \left( -\alpha \beta {q}^{n}+\gamma \right) }{ \left( \alpha \beta {q}^{2 n}-1 \right) \left( \alpha \beta {q}^{2 n-1}-1 \right) } \tilde{P}_{n-1}(x;\alpha,\beta,\gamma;q).\label{eq5bqj}
\end{gather}
\end{subequations}
\end{Proposition}

\begin{Corollary}
 \begin{gather}
\tilde{P}_n\left(x;\frac{\alpha}{q},\frac{\beta}{q},\gamma;q\right) =\tilde{P}_n(x;\alpha,\beta,\gamma;q) -\frac {\a q\left( \alpha\beta{q}^{2 n}-\beta{q}^{n+1
}-\beta{q}^{n}+q \right) \left( {q}^{n}-1 \right) \left(\gamma{q}^{n}-1 \right) }{ \left( \alpha \beta {q}^{2 n}-1
 \right) \left( \alpha \beta {q}^{2 n}-{q}^{2} \right) }\nonumber\\
 \hphantom{\tilde{P}_n\left(x;\frac{\alpha}{q},\frac{\beta}{q},\gamma;q\right) =}{}\times \tilde{P}_{n-1}(x;\alpha,\beta,\gamma;q)\nonumber\\
 \hphantom{\tilde{P}_n\left(x;\frac{\alpha}{q},\frac{\beta}{q},\gamma;q\right) =}{}
 -\frac {\a^2 \beta \left( {q}^{n}-1 \right) \left( \beta{q}^{n}-q \right) \left( \alpha{q}^{n}-q \right) \left( \gamma{q}^{n}-1 \right) \left(\gamma{q}^{n}-q \right) \left( {q}^{n}-q \right){q}^{n+3} }{ \left( \alpha \beta {q}^{2 n}-{q}^{2} \right) ^{2} \left( \alpha \beta {q}^{2 n}-{q}^{3} \right) \left( \alpha \beta {q}^{2 n}-q \right) }\nonumber\\
 \hphantom{\tilde{P}_n\left(x;\frac{\alpha}{q},\frac{\beta}{q},\gamma;q\right) =}{} \times \tilde{P}_{n-2}(x;\alpha,\beta,\gamma;q).\label{eq6bqj}
\end{gather}
\end{Corollary}

\begin{proof} By replacing $\b$ with $\frac{\b}{q}$ in \eqref{eq1bqj}, we obtain an equation involving polynomials \linebreak $\tilde{P}_n\big(x;\frac{\a}{q},\frac{\b}{q},\gamma;q\big)$, $\tilde{P}_n\big(x;\a,\frac{b}{q},\gamma;q\big)$ and $\tilde{P}_{n-1}\big(x;\a,\frac{\b}{q},\gamma;q\big)$. We use~\eqref{eq3bqj} to replace the latter two polynomials and, after simplifying, we obtain~\eqref{eq6bqj}.
\end{proof}

We will start by proving the quasi-orthogonality of the sequence $\big\{\tilde{P}_{n}\big(x;\frac{\a}{q^k},{\beta},\gamma;q\big)\big\}_{n=0}^{\infty}$. In order to ensure that the parameter $\frac{\a}{q^k}$ is not in the region where orthogonality is guaranteed, we fix $\a>1$ with $0<\a q<1$, such that $\frac{\a}{q^k}>1$, $k\in\{1,2,\dots,n-1\}$.

\begin{Theorem}\label{Thmbj} Let $k,l,m\in\nn_0$, $\a, \b,\gamma\in \rr$, $0<\a q<1$, $0\leq \beta q<1$ and $\gamma<0$. The sequence of big $q$-Jacobi polynomials
\begin{itemize}\itemsep=0pt
\item[$(i)$]
$\big\{\tilde{P}_{n}\big(x;\frac{\a}{q^k},{\beta},\gamma;q\big)\big\}_{n=0}^{\infty}$, $\a>1$, is quasi-orthogonal of order $k\le n-1$ with respect to $w(x)$ on the interval $(\g q,\a q)$ and the polynomials have at least $(n-k)$ real, distinct zeros in $(\g q,\a q)$;
\item[$(ii)$]
$\big\{\tilde{P}_{n}\big(x;\a,\frac{\b}{q^m},\gamma;q\big)\big\}_{n=0}^{\infty}$, $\b>1$, is quasi-orthogonal of order $m \le n-1$ with respect to $w(x)$ on $(\g q,\a q)$ and the polynomials have at least $(n-m)$ real, distinct zeros in $(\g q,\a q)$;
\item[$(iii)$]
$\big\{\tilde{P}_{n}\big(x;\a,\frac{\b}{q^l},\frac{\gamma}{q^l};q\big)\big\}_{n=0}^{\infty}$, $\b>1$, is quasi-orthogonal of order $l \le n-1$ with respect to $w(x)$ on $(\g q,\a q)$ and the polynomials have at least $(n-l)$ real, distinct zeros in $(\g q,\a q)$;
\item[$(iv)$]
$\big\{\tilde{P}_{n}\big(x;\frac{\a}{q^k},\frac{\beta}{q^m},\gamma;q\big)\big\}_{n=0}^{\infty}$, $\a,\b>1$ is quasi-orthogonal of order $k+m\le n-1$ with respect to~$w(x)$ on $(\g q,\a q)$ and the polynomials have at least $n-(k+m)$ real, distinct zeros in $(\g q,\a q)$.
\end{itemize}
\end{Theorem}
\begin{proof} (i) Fix $\a>1$ such that $0<\alpha q<1$. From Lemma~\ref{Brez} and \eqref{eq1bqj}, it follows that $\tilde{P}_{n}\big(x;\frac{\a}{q},{\beta},\gamma;q\big)$ is quasi-orthogonal of order one on $(\g q, \a q)$ and according to Lemma~\ref{location}, at least $(n-1)$ zeros of $\tilde{P}_{n}\big(x;\frac{\a}{q},{\beta},\gamma;q\big)$ lie in the interval~$(\g q,\a q)$. By iteration, we can express $\tilde{P}_{n}\big(x;\frac{\a}{q^k},{\b},\gamma;q\big)$ as a linear combination of $\tilde{P}_{n}(x;{\a},{\beta},\gamma;q)$, $\tilde{P}_{n-1}(x;{\a},{\beta},\gamma;q)$, $\dots$, $\tilde{P}_{n-k}(x;{\a},{\beta},\gamma;q)$, and from Lemma~\ref{Brez} we deduce that $\tilde{P}_{n}\big(x;\frac{\a}{q^k},{\beta},\gamma;q\big)$ is quasi-orthogonal of order~$k$ on $(\g q, \a q)$. It follows from Lemma~\ref{location} that at least $(n-k)$ zeros of $\tilde{P}_{n}\big(x;\frac{\a}{q^k},{\beta},\gamma;q\big)$ are in~$(\g q,\a q)$.

(ii)--(iii)
Fix $\b>1$ such that $0<\b q<1$. The proofs follow in exactly the same way as the proof of~(i), by using \eqref{eq3bqj} and~\eqref{eq5bqj}, together with Lemmas~\ref{Brez} and~\ref{location}.

(iv) Fix $\a>1$, $\b>1$ such that $0<\alpha q<1$ and $0< \beta q<1$. From \eqref{eq6bqj}, $\tilde{P}_{n}\big(x;\frac{\a}{q},\frac{\beta}{q},\gamma;q\big)$ can be written as a linear combination of $\tilde{P}_{n}(x;{\a},{\beta},\gamma;q)$, $\tilde{P}_{n-1}(x;{\a},{\beta},\gamma;q)$ and $\tilde{P}_{n-2}(x;{\a},{\beta},\gamma;q)$, and it follows from Lemma~\ref{Brez} that each polynomial $\tilde{P}_{n}\big(x;\frac{\a}{q},\frac{\beta}{q},\gamma;q\big)$, $n\in\{1,2,\dots\}$, is quasi-orthogonal of order two on $(\g q, \a q)$. From Lemma~\ref{location}, we know that at least $(n-2)$ zeros of $\tilde{P}_{n}\big(x;\frac{\a}{q},\frac{\beta}{q},\gamma;q\big)$ lie in $(\g q,\a q)$. By iteration, we can express $\tilde{P}_{n}\big(x;\frac{\a}{q^k},\frac{\beta}{q^m},\gamma;q\big)$ as a linear combination of $\tilde{P}_{n}(x;{\a},{\beta},\gamma;q)$, $\tilde{P}_{n-1}(x;{\a},{\beta},\gamma;q),\dots,\tilde{P}_{n-(k+m)}(x;{\a},{\beta},\gamma;q)$, and the results follow directly from Lemmas~\ref{Brez} and~\ref{location}.
\end{proof}

In order to determine the location of the zeros of the order one and order two quasi-orthogonal systems, we use a $q$-analogue of the Vandermonde identity~\cite[equation~(1.2.9)]{GR}, namely~\cite[equation~(1.5.3)]{GR}
\begin{gather}\label{vdM}
\qhypergeom{2}{1}{q^{-n},b}{c}{q;q}=\frac{(\frac cb;q)_n }{(c;q)_n} b^n.
\end{gather}

\begin{Theorem}\label{Th}
Let $n\in\nn$, $\a, \b,\gamma\in \rr$, such that $0<\alpha q, \b q<1$ and $\g<0$. Suppose $x_{n,j}$, $j\in\{1,2,\dots,n\}$ denote the zeros of $\tilde{P}_{n}(x;\a,\b,\g;q)$, $y_{n,j}$, $j\in\{1,2,\dots,n\}$ the zeros of $\tilde{P}_{n}\big(x;\frac{\a}{q},\b,\g;q\big)$, $z_{n,j}$, $j\in\{1,2,\dots,n\}$ the zeros of $\tilde{P}_{n}\big(x;\a,\frac{\b}{q},\g;q\big)$, $v_{n,j}$, $j\in\{1,2,\dots,n\}$ the zeros of $\tilde{P}_{n}\big(x;\frac{\a}{q},\frac{\b}{q},\g;q\big)$ and $w_{n,j}$, $j\in\{1,2,\dots,n\}$ the zeros of $\tilde{P}_{n}\big(x;{\a},\frac{\b}{q},\frac{\g}{q};q\big)$. Then,
\begin{itemize}\itemsep=0pt
\item[$(i)$] when we fix $\a>1$, such that $0<\a q<1$,
\begin{gather*} \g q<x_{n,1}<y_{n,1}<x_{n-1,1}<x_{n,2}<y_{n,2}<x_{n-1,2}<\cdots<x_{n-1,n-1}<x_{n,n}<y_{n,n};
\end{gather*}
\item[$(ii)$] when we fix $\b>1$, such that $0<\b q<1$,
\begin{gather*} z_{n,1}<\g q<x_{n,1}<x_{n-1,1}<z_{n,2}<x_{n,2}<\cdots<x_{n-1,n-1}<z_{n,n}<x_{n,n}<\a q;
\end{gather*}
\item[$(iii)$] when we fix $\b>1$, such that $0<\b q<1$, we also have
\begin{gather*} w_{n,1}<x_{n,1}<x_{n-1,1}<w_{n,2}<x_{n,2}<\cdots<x_{n-1,n-1}<w_{n,n}<x_{n,n}<\a q;
\end{gather*}
\item[$(iv)$] when we fix $\a,\b>1$, such that $0<\a q,\b q<1$, all the zeros of $\tilde{P}_{n}\big(x;\frac{\a}{q},\frac{\b}{q},\g;q\big)$ are real and distinct and $v_{n,1}<\g q$.
\end{itemize}
\end{Theorem}
\begin{proof}
(i) From \eqref{eq1bqj}, we obtain
\begin{gather*}
a_n={\frac {\a q \left( {q}^{n}-1 \right) \left( \beta{q}^{n}-1 \right)
 \left( \gamma {q}^{n}-1 \right) }
 { \left( \alpha\beta{q}^{2n}-q \right) \left( \alpha\beta{q}^{2n}-1 \right) }}<0,
\end{gather*} and the interlacing result, as well as the location of $y_{n,1}$, follows from Lemma~\ref{Joulak2}(i).

(ii) From \eqref{eq3bqj}, we obtain
\begin{gather*}
a_n=-\frac { \alpha \beta {q}^{n+1}\left( {q}^{n}-1 \right) \left( \alpha {q}^{n}-1 \right)\left( \gamma {q}^{n}-1 \right) }{ \left( \alpha \beta {q}^{2 n}-1 \right)
 \left( \alpha \beta {q}^{2n}-q \right) },
 \end{gather*}
 which is positive when taking into consideration the values of the parameters. The interlacing result, as well as the location of $y_{n,n}$, follows from Lemma~\ref{Joulak2}(ii).

 The polynomial $\tilde{P}_{n}(x;\a,{\b},{\gamma};q)$ evaluated at $x=\g q$, can be written in terms of a $_2 \phi_1$-hypergeometric function. We apply~(\ref{vdM}), and simplify, to obtain
 \begin{gather}\label{fn}f_n(\g q)= \frac{\tilde{P}_{n}(\g q;\a,{\b},{\gamma};q)}{\tilde{P}_{n-1}(\g q;\a,{\b},{\gamma};q)}={\frac {\alpha \left( \alpha\beta{q}^{n}-1 \right)
 \left( \beta{q}^{n}-1 \right) \left( \gamma{q}^{n}-1 \right){q}^{n+1} }{
 \left( \alpha\beta{q}^{2n}-q \right) \left( \alpha\beta{q}
^{2n}-1 \right) }}\end{gather}
 and by taking into account the values of the parameters, this expression is negative. Since
\begin{gather*}
-a_n- f_n(\g q)=-{\frac {\alpha \left( \beta-1 \right) \left( \gamma {q}^
{n}-1 \right){q}^{n+1} }{\alpha\beta{q}^{2n}-q}}<0,
\end{gather*} the result follows from Lemma~\ref{Joulak1}(i).

(iii) From \eqref{eq5bqj}, we obtain
\begin{gather*}
a_n=-{\frac { \left( {q}^{n}-1 \right) \left( \alpha {q}^{n}-1 \right)
 \left( -\alpha\beta {q}^{n}+\gamma \right) q}{ \left( \alpha \beta {q}^{2n}-q \right) \left( \alpha\beta {q}^{2n}-1 \right) }}>0,
 \end{gather*} and the interlacing result, as well as the location of~$w_{n,n}$, follows from Lemma~\ref{Joulak2}(ii).

 (iv) Fix $\a>1$ and $\b>1$ such that $0<\alpha q<1$ and $0< \beta q<1$. We use~\eqref{eq6bqj}, with $a_n$ the coefficient of $\tilde{P}_{n-1}(x;\alpha,\beta,\gamma;q)$ and $b_n$ the coefficient of $\tilde{P}_{n-2}(x;\alpha,\beta,\gamma;q)$. By taking into account the values of the parameters,
\begin{gather*}
b_n= -\frac {\a^2 \beta \left( {q}^{n}-1 \right) \left( \beta{q}^{n}-q \right) \left( \alpha{q}^{n}-q \right) \left( \gamma{q}^{n}-1 \right) \left(\gamma{q}^{n}-q \right) \left( {q}^{n}-q \right){q}^{n+3} }{ \left( \alpha \beta {q}^{2 n}-{q}^{2} \right) ^{2} \left( \alpha \beta {q}^{2 n}-{q}^{3} \right) \left( \alpha \beta {q}^{2 n}-q \right) }<0,
\end{gather*}
and it follows from \cite[Theorem~4]{BDR} that $v_{n,j}$, $j\in\{1,2,\dots,n\}$, are real.

In order to determine the location of $v_{n,1}$ and $v_{n,n}$, we use \cite[Theorem 9]{Joulak_2005}. Since
\begin{gather*}f_n(\g q) f_{n-1}(\g q)+a_nf_{n-1}(\g q)+b_n = \frac{\tilde{P}_{n}(\g q;\alpha,\beta,\gamma;q)}{\tilde{P}_{n-1}(\g q;\alpha,\beta,\gamma;q)}+a_n\frac{\tilde{P}_{n-1}(\g q;\alpha,\beta,\gamma;q)}{\tilde{P}_{n-2}(\g q;\alpha,\beta,\gamma;q)}+b_n\\
\hphantom{f_n(\g q) f_{n-1}(\g q)+a_nf_{n-1}(\g q)+b_n}{}
 = {\frac { {\alpha}^{2} \left( \beta-1 \right) \left( \gamma{q}^{n}-
1 \right) \left(\b {q}^{n}-q \right) \left( \gamma{
q}^{n}-q \right){q}^{n+2} }{ \left( \alpha \beta {q}^{2n}
-{q}^{2} \right) \left( \alpha \beta {q}^{2n}-{q}^{3} \right) }} < 0,
\end{gather*}
it follows that $v_{n,1}<\g q$. Furthermore,
\begin{gather*} f_n(\a q)f_{n-1}(\a q)+a_nf_{n-1}(\a q)+b_n\\
\qquad {} =-{\frac { \left( \alpha{q}^{n}-q \right) \left( {\alpha}^
{2}\beta {q}^{2n} -\alpha \gamma {q}^{n+1}-\a \beta{q}^{n}-\alpha{q}^{n+1}+
\gamma{q}^{n+1}+\alpha q \right) \left(\gamma -\alpha \beta \right){q}^{n+2} }{ \left(\a \beta{q}^{2n}-{q}^{3} \right)
 \left(\a \beta{q}^{2n}-{q}^{2} \right) }}\end{gather*}
and since the sign of this expression varies as the parameters vary within the regions applicable, we cannot determine the position of~$v_{n,n}$.
\end{proof}

\begin{Remark}\quad
\begin{itemize}\itemsep=0pt
\item [(i)]
From Theorem~\ref{Thmbj}(i) we know that the polynomial $\tilde{P}_{n}\big(x;\frac{\a}{q},{\beta},\gamma;q\big)$, $\a>1$, is quasi-orthogonal of order one and an interlacing result is proved in Theorem~\ref{Th}(i), but the location of the extreme zero~$y_{n,n}$, with respect to $(\g q,\a q)$, is not fixed, since the sign of
\begin{gather}-a_n- f_n(\a q) = a_n-\frac{\tilde{P}_{n}(\a q;{\a},{\beta},\gamma;q)}{\tilde{P}_{n-1}(\a q;{\a},{\beta},\gamma;q)}\nonumber\\
\hphantom{-a_n- f_n(\a q)}{} =-{\frac { \left( -{\alpha}^{2}{q}^{2n}\beta+\alpha \beta {q}^{n}+{
q}^{n}\alpha \gamma+\alpha {q}^{n}-\gamma {q}^{n}-\alpha \right) q
}{\alpha \beta {q}^{2n}-q}}\label{use}
\end{gather}
changes as the parameters vary within the region applicable.
\item [(ii)]
When $\b=0$ in the definition of the big $q$-Jacobi polynomials~\eqref{eq:BigqJacobi}, we obtain the big $q$-Laguerre polynomials, i.e., $\tilde{P}_n(x;\alpha,0,\g;q)=\tilde{P}_n(x;\a,\g;q)$ \cite[equation~(14.5.13)]{KLS} and we can use~\eqref{eq1bqj} with $\b=0$. Let $x_{n,j}$, $j\in\{1,2,\dots,n\}$, be the zeros of $\tilde{P}_{n}(x;\a,\g;q)$ and~$y_{n,j}$, $j\in\{1,2,\dots,n\}$, the zeros of $\tilde{P}_{n}\big(x;\frac{\a}{q},\g;q\big)$.
When $\b=0$ in~(\ref{use}), we obtain
\begin{gather*} -a_n- f_n(\a q)=a_n-\frac{\tilde{P}_{n}(\a q;{\a},\gamma;q)}{\tilde{P}_{n-1}(\a q;{\a},\gamma;q)} =\g {q}^{n}(\a-1)+\alpha \big({q}^{n}-1\big)<0,
\end{gather*}
taking into consideration that $\a>1$, $0<\a q <1$ and $\g<0$, and
\begin{gather*} f_n(\g q)+a_n=\frac{\tilde{P}_{n}(\g q;{\a},\gamma;q)}{\tilde{P}_{n-1}(\g q;{\a},\gamma;q)}+a_n
=\alpha\ {q}^{n} \big( \gamma {q}^{n}-1 \big) +a_n =\a \big( \gamma {q}^{n}-1 \big)<0.
\end{gather*}
We thus have $f_n(\g q)<-a_n<f_n(\a q)$ and according to Lemma~\ref{Joulak1}(iii), all the zeros of the order one quasi-orthogonal polynomial $\tilde{P}_{n}\big(x;\frac{\a}{q},\g;q\big)$, $\a>1$, lie in $(\g q,\a q)$. Furthermore,
since $a_n<0$, it follows from Lemma~\ref{Joulak2}(ii) that
\begin{gather*} \g q<x_{n,1}<y_{n,1}<x_{n-1,1}<x_{n,2}<y_{n,2}<x_{n-1,2}<\cdots<x_{n-1,n-1}\\
\hphantom{\g q}{} <x_{n,n}<y_{n,n}<\a q.\end{gather*}
\end{itemize}
\end{Remark}

\subsection[The $q$-Hahn polynomials]{The $\boldsymbol{q}$-Hahn polynomials}
The $q$-Hahn polynomials
 \begin{gather*}\tilde{Q}_n(\bar{x};\alpha,\beta,N|q)=\frac{(\alpha q;q)_n(q^{-N};q)_n}{(\alpha\beta q q^n;q)_n}\,\qhypergeom{3}{2}{q^{-n},\alpha\beta q^{n+1},\bar{x}}{\alpha q,q^{-N}}{q;q},\end{gather*}
 with $\bar{x}=q^{-x}$ and $n\in\{0,1,\ldots,N\}$ are orthogonal on $\big(1,q^{-N}\big)$ with respect to the discrete weight $w(x)=\frac{(\a q,q^{-N};q)_x }{(q,\b^{-1}q^{-N};q)_x(\a \b q)^x}$ for $0<\alpha q<1$ and $0<\beta q<1$ or $\alpha >q^{-N}$ and $\beta >q^{-N}$. We recall the definition of quasi-orthogonality with respect to a discrete weight (cf.~\cite{chiharaboek}). A~polynomial $Q_{n,k}$ of exact degree $n \ge k$, $n\in\{0,1,\dots,N\}$, where $N$ may be infinite, is discrete quasi-orthogonal of order $k$ with $w_i$ the weight at each point $x_i$, $i\in\{0,1,\dots,N-1\}$, if
\begin{gather*} \sum_{i=0}^{N-1} (x_i)^m Q_{n,k}(x_i)w_i \ \begin{cases}=0,& \mbox{for} \ m\in\{0,1,\dots,n-k-1\},\\ \neq0, & \mbox{for} \ m=n-k.\end{cases}\end{gather*}

 We will consider the case $0<\alpha q, \beta q<1$. The following equations follow from \cite[equations~(8a) and~(8b)]{KoepfJT_2017}, with $\a$ and $\b$ replaced by $\frac{\a}{q}$ and $\frac{\b}{q}$, respectively,
\begin{subequations}
 \begin{gather}
\tilde{Q}_n\left(\bar{x};\frac{\alpha}{q},\beta,N|q\right)= \tilde{Q}_n(\bar{x};\alpha,\beta,N|q) \nonumber\\
\hphantom{\tilde{Q}_n\left(\bar{x};\frac{\alpha}{q},\beta,N|q\right)=}{} +{\frac {\alpha \left( {q}^{n}-1 \right) \left(\beta q^n-1 \right) \left( q^n-q^{N+1} \right) }{ {q}^{N}\left( \alpha\beta {q}^{2 n}-1 \right) \left(\alpha \beta {q}^{2 n}-q \right) }\tilde{Q}_{n-1}(\bar{x};\alpha ,\beta,N|q)};\label{eq1qh} \\
 \tilde{Q}_n\left(\bar{x};\alpha,\frac{\beta}{q},N|q\right) =\tilde{Q}_n(\bar{x};\alpha,\beta ,N|q)\nonumber\\
 \hphantom{\tilde{Q}_n\left(\bar{x};\alpha,\frac{\beta}{q},N|q\right) =}{} -\frac {\alpha \beta \left(q^n
-q^{N+1} \right) (\alpha q^n -1) ( q^n-1)}{ \left(\alpha \beta{q}^{2 n}-1 \right) \left(\alpha \beta {q}^{2 n}-q \right) {q}^{N-n} }\tilde{Q}_{n-1}(\bar{x};\alpha,\beta ,N|q) .\label{eq2qh}
 \end{gather}
\end{subequations}

\begin{Corollary}
\begin{gather}
\tilde{Q}_{n}\left(\bar{x};\frac{\alpha}{q},\frac{\beta}{q} ,N|q\right) =\tilde{Q}_{n}(\bar{x};\alpha,\beta ,N|q)\nonumber\\
\qquad{} -\frac {\alpha \left( \alpha \beta{q}^{2 n}-\beta{q}^{n+1}-\beta{q}^{n}+q \right) \left( {q}^{n}-1 \right) \left({q}^{n}-{q}^{N+1} \right)}{ \left( \alpha \beta {q}^{2 n}-{q}^{2} \right) \left( \alpha \beta {q}^{2 n}-1 \right){q}^{N} } \tilde{Q}_{n-1}(\bar{x};\alpha,\beta ,N|q)\nonumber\\
\qquad{} -\frac {{\alpha}^{2}\beta q^{n+1} \left( q^n-1 \right) \left(\beta q^n-q \right) \left({q}^{n} -{q}^{N+1} \right) \left(\alpha{q}^{n}-q \right) \left( {q}^{n}-q \right) \left({q}^{n} -{q}^{N+2} \right) }{ \left( \alpha \beta {q}^{2 n}-{q}^{2} \right) ^{2} \left( \alpha \beta {q}^{2 n}-q \right) \left( \alpha \beta {q}^{2 n}-{q}^{3} \right)q^{2N} }\nonumber\\
\qquad{}\times \tilde{Q}_{n-2}(\bar{x};\alpha,\beta ,N|q).\label{eq7qh}
 \end{gather}
\end{Corollary}

\begin{Theorem}\label{Thmqhqo}
Let $k,m, N\in\nn_0$, $n\in\{0,1,2,\ldots,N\}$, $\a,\b,\in \rr$. For $0<\a q<1$ and \smash{$0< \beta q<1$}, the sequence of $q$-Hahn polynomials
\begin{itemize}\itemsep=0pt
\item[$(i)$] $\big\{\tilde{Q}_{n}\big(\bar{x};\frac{\alpha}{q^k},\beta ,N|q\big)\big\}_{n=0}^{N}$, with $\a>1$, is quasi-orthogonal of order $k\le n-1$ with respect to the discrete weight $w(x)$ on the interval $\big(1,q^{-N}\big)$ and the polynomials have at least $(n-k)$ real, distinct zeros in $\big(1,q^{-N}\big)$;
\item[$(ii)$] $\big\{\tilde{Q}_{n}\big(\bar{x};\a,\frac{\b}{q^m},N|q\big)\big\}_{n=0}^{N}$, $\b>1$, is quasi-orthogonal of order $m\le n-1$ with respect to $w(x)$ on $\big(1,q^{-N}\big)$ and the polynomials have at least $(n-m)$ real, distinct zeros in the interval $\big(1,q^{-N}\big)$;
\item[$(iii)$] $\big\{\tilde{Q}_{n}\big(\bar{x};\frac{\alpha}{q^k},\frac{\beta}{q^m} ,N|q\big)\big\}_{n=0}^{N}$, $\a,\b>1$, is quasi-orthogonal of order $k+m\le n-1$ with respect to $w(x)$ on $\big(1,q^{-N}\big)$ and the polynomials have at least $n-(k+m)$ real, distinct zeros in $\big(1,q^{-N}\big)$.
\end{itemize}
\end{Theorem}
\begin{proof} (i) Fix $\a>1$, $\b\in\rr$, such that $0<\alpha q<1$, $0<\beta q<1$. From Lemma~\ref{Brez} and~\eqref{eq1qh}, it follows that $\tilde{Q}_{n}\big(\bar{x};\frac{\alpha}{q},\beta ,N|q\big)$ is quasi-orthogonal of order one on $\big(1,q^{-N}\big)$. From Lemma~\ref{location} we know that at least $(n-1)$ zeros of $\tilde{Q}_{n}\big(\bar{x};\frac{\alpha}{q},\beta ,N|q\big)$ lie in the interval $\big(1,q^{-N}\big)$. By iteration, we can express $\tilde{Q}_{n}\big(\bar{x};\frac{\alpha}{q^k},\beta ,N|q\big)$ as a linear combination of $\tilde{Q}_{n}(\bar{x};{\alpha},\beta ,N|q)$, $\tilde{Q}_{n-1}(\bar{x};\alpha,\beta ,N|q),\dots$, $\tilde{Q}_{n-k}(\bar{x};\alpha,\beta ,N|q)$, and the results follow from Lemmas~\ref{Brez} and~\ref{location}.

(ii) Fix $\b>1$, $\a\in\rr$, such that $0<\alpha q<1$, $0<\beta q<1$. The quasi-orthogonality follows in the same way as in~(i), by using~\eqref{eq2qh}.

(iii) Fix $\a>1$ and $\b>1$ such that $0<\alpha q<1$ and $0< \beta q<1$. From~\eqref{eq7qh}, we see that $\tilde{Q}_{n}\big(\bar{x};\frac{\alpha}{q},\frac{\beta}{q} ,N|q\big)$ can be written as a linear combination of $\tilde{Q}_{n}(\bar{x};{\alpha},\beta ,N|q)$, $\tilde{Q}_{n-1}(\bar{x};\alpha,\beta ,N|q)$ and $\tilde{Q}_{n-2}(\bar{x};{\alpha},\beta ,N|q)$ and it follows from Lemma~\ref{Brez} that the sequence $\tilde{Q}_{n}\big(\bar{x};\frac{\a}{q},\frac{\beta}{q},\gamma;q\big)$ is quasi-orthogonal of order two on $\big(1,q^{-N}\big)$. By iteration, we can express $\tilde{Q}_{n}\big(\bar{x};\frac{\alpha}{q^k},\frac{\beta}{q^m} ,N|q\big)$ as a~linear combination of $\tilde{Q}_{n}(\bar{x};{\alpha},\beta ,N|q)$, $\tilde{Q}_{n-1}(\bar{x};{\alpha},\beta ,N|q)$, $\dots$, $\tilde{Q}_{n-(k+m)}(\bar{x};{\alpha},\beta ,N|q)$, and the result follows directly from Lemma~\ref{Brez}. It follows from Lemma~\ref{location} that at least $n-(k+m)$ zeros of $\tilde{Q}_{n}\big(\bar{x};\frac{\a}{q^k},\frac{\beta}{q^m},\gamma;q\big)$ lie in the interval $\big(1,q^{-N}\big)$.
\end{proof}

\begin{Theorem}\label{Thmqhalpha}
Let $N\in\nn_0$, $n\in\{0,1,2,\ldots,N\}$, $\a,\b\in \rr$, $0<\a q,\b q<1$, and let $x_{n,j}$, $j\in\{1,2,\dots,n\}$, denote the zeros of $\tilde{Q}_{n}(\bar{x};\alpha,\beta ,N|q)$, $y_{n,j}$, $j\in\{1,2,\dots,n\}$, the zeros of $\tilde{Q}_{n}\big(\bar{x};\frac{\alpha}{q},\beta ,N|q\big)$ and $z_{n,j}$, $j\in\{1,2,\dots,n\}$, the zeros of $\tilde{Q}_{n}\big(\bar{x};\a,\frac{\b}{q},N|q\big)$. Then
\begin{itemize}\itemsep=0pt
\item[$(i)$] if $\a>1$, $y_{n,1}<1<x_{n,1}<x_{n-1,1}<y_{n,2}<x_{n,2}<\cdots<x_{n-1,n-1}<y_{n,n}<x_{n,n}<q^{-N}$;
\item[$(ii)$] if $\b>1$, $1<x_{n,1}<z_{n,1}<x_{n-1,1}<x_{n,2}<z_{n,2}<\cdots<x_{n-1,n-1}<x_{n,n}<q^{-N}<z_{n,n}$.
\end{itemize}
\end{Theorem}

\begin{proof} (i) From \eqref{eq1qh} we obtain the value
\begin{gather*}
a_n={\frac {\alpha \left( {q}^{n}-1 \right) \left(\beta{q}^{n}-1 \right) \left( {q}^{n}-{q}^{N+1} \right) }{ {q}^{N}\left( \alpha\beta{q}^{2n}-1 \right) \left(\alpha\beta{q}^{2n}-q \right) }},
\end{gather*} which is positive when we take into consideration the values of the parameters. The interlacing result, from which we can deduce the location of $y_{n,n}$, follows from Lemma~\ref{Joulak2}(ii).

In order to prove that $y_{n,1}$ does not lie in the interval of orthogonality, i.e., $y_{n,1}<1$, we use the fact that
\begin{gather*}
\qhypergeom{3}{2}{q^{-n},\alpha\beta q^{n+1},1}{\alpha q,q^{-N}}{q;q}=\qhypergeom{3}{2}{q^{-n},\alpha\beta q^{n+1},q^{-0}}{\alpha q,q^{-N}}{q;q}=1
\end{gather*} and Lemma \ref{Joulak1}.
Consider
\begin{gather*} f_n(1)=\frac{\tilde{Q}_{n}(1;\alpha,\beta ,N|q)}{\tilde{Q}_{n-1}(1;\alpha,\beta,N|q)}
=-{\frac { \left( \alpha{q}^{n}-1 \right) \left( \alpha\beta{q}^
{n}-1 \right) \left({q}^{n}-{q}^{N+1} \right) }{ \left( \alpha\beta{q}^{2n}-1 \right) \left(\alpha\beta{q}^{2n}-q \right) }},
\end{gather*} which is negative for the appropriate parameter values. We thus have
\begin{gather*} -a_n-f_n(1)={\frac { \left( \alpha-1 \right) \left( {q}^{n}-{q}^{N+1}\right) }{
 \left( \alpha\beta {q}^{2n} -q \right) {q}^{N}}}<0,
 \end{gather*}
i.e., $-a_n<f_n(1)<0$, and the result follows from Lemma~\ref{Joulak1}(i).

(ii) From \eqref{eq2qh} we obtain
\begin{gather*}
a_n=-\frac {\alpha\beta \left( {q}^{n}
-{q}^{N+1} \right) \left(\alpha {q}^{n} -1\right) \left( {q}^{n}-1 \right){q}^{n} }{ \left(\alpha \beta {q}^{2n}-1 \right) \left(\alpha\beta {q}^{2n}-q \right) {q}^{N} },\end{gather*} which is negative. The interlacing result, from which we can deduce the location of $z_{n,1}$, follows from Lemma~\ref{Joulak2}(i).

The polynomial $\tilde{Q}_{n}(\bar{x};{\alpha},\beta ,N|q)$ evaluated at $\bar{x}=q^{-N}$, can be written in terms of a $_2 \phi_1$-hypergeometric function. We apply (\ref{vdM}), and simplify, to obtain
\begin{gather*} f_n\big(q^{-N}\big)=\frac{\tilde{Q}_{n}\left(q^{-N};\alpha,\beta ,N|q\right)}{\tilde{Q}_{n-1}\left(q^{-N};\alpha,\beta ,N|q\right)}= -{\frac {\a \left( \beta {q}^{n}-1 \right) \left( \alpha
\beta {q}^{n}-1 \right) \left( -{q}^{N+1}+{q}^{n} \right) q^n}{
 \left( \alpha\beta {q}^{2n}-q \right) \left( \alpha\beta {q}
^{2n}-1 \right)q^N }}.\end{gather*}
 When taking into consideration the values of the parameters,
\begin{gather*} -a_n- f_n\big(q^{-N}\big)=-{\frac {\a \left( \beta-1 \right) \left( {q}^{n}-{q}^{N+1}
 \right){q}^{n} }{(\alpha\beta {q}^{2n}-q)q^N}}<0
 \end{gather*} and the result follows from Lemma~\ref{Joulak1}(ii).
\end{proof}

\begin{Remark}We cannot say anything about the location of the zeros of $\tilde{Q}_{n}\big(\bar{x};\frac{\alpha}{q^2},\beta ,N|q\big)$, since the coefficient of $\tilde{Q}_{n-2}(\bar{x};{\alpha},\beta ,N|q)$, in the equation
\begin{gather*}
 \tilde{Q}_n\left(\bar{x};\frac{\alpha}{q^2},\beta ,N|q\right)=\tilde{Q}_{n}(\bar{x};\alpha,\beta ,N|q)\\
 {} +\frac {\alpha \left( q+1 \right)
\left( {q}^{n}-1 \right) \left( \beta{q}^{n}-1 \right) \left({q}^{n}-{q}^{N+1} \right)}{ \left( \alpha\beta{q
}^{2n}-{q}^{2} \right) \left( \alpha\beta{q}^{2n}-1 \right) {q}^{N}}\tilde{Q}_{n-1}(\bar{x};\alpha,\beta ,N|q) \\
 {} +\frac {(\alpha q)^{2} \left( {q}^{n}-1 \right) \left( \beta{q}^{n}-q \right) \left({q}^{n}-{q}^{N+1} \right) \left( {q}^{n}-q \right) \left( \beta{q}^{n}-1 \right) \left({q}^{n} -{q}^{N+2} \right) }{ \left( \alpha\beta{q}^{2n}-{q}^{2} \right) ^{2} \left(\alpha\beta{q}^{2n}-q \right) \left( \alpha\beta{q}^{2n}-{q}^{3} \right){q}^{2 N} }\tilde{Q}_{n-2}(\bar{x};\alpha,\beta ,N|q),
\end{gather*}
that can be obtained from \eqref{eq1qh}, is positive (cf.\ \cite[Theorem~4]{BDR}). The same is true for the location of the zeros of $\tilde{Q}_{n}\big(\bar{x};\a,\frac{\b}{q^2} ,N|q\big)$ and the equation can be found in the accompanying Maple file.
\end{Remark}

\begin{Theorem}\label{Thmqh4} Let $N\in\nn_0$, $n\in\{0,1,2,\ldots,N\}$, $\a,\b>1$. All the zeros of $\tilde{Q}_{n}\big(\bar{x};\frac{\alpha}{q},\frac{\beta}{q} ,N|q\big)$ are real and distinct and if $z_{n,j}$, $j\in\{1,2,\dots,n\}$, are the zeros of $\tilde{Q}_{n}\big(\bar{x};\frac{\alpha}{q},\frac{\beta}{q} ,N|q\big)$, then $z_{n,1}<1$ and $q^{-N}<z_{n,n}$.
 \end{Theorem}
\begin{proof} Fix $\a>1$ and $\b>1$ such that $0<\alpha q<1$ and $0< \beta q<1$. We use~\eqref{eq7qh} with $a_n$ the coefficient of $\tilde{Q}_{n-1}(\bar{x};{\alpha},\beta ,N|q)$ and $b_n$ the coefficient of $\tilde{Q}_{n-2}(\bar{x};{\alpha},\beta ,N|q)$. By taking into account the values of the parameters, we see that
\begin{gather*}
b_n= -{\frac { {\alpha}^{2}\b \left( \beta {q}^{n}-q \right) \left( {q}^{n}-{q}^{
N+1} \right) \left( {q}^{n}-{q}^{N+2} \right) \left( \alpha
 {q}^{n}-q \right) \left( {q}^{n}-1 \right) \left( {q}^
{n}-q \right){q}^{n+1} }{ \left( \alpha\beta {q}^{2n}-q \right)
 \left( \alpha\beta {q}^{2n}-{q}^{3} \right) \left( \alpha \beta {q}^{2n}-{q}^{2} \right) ^{2}{q}^{2 N}}}<0,
 \end{gather*}
and it follows from \cite[Theorem~4]{BDR} that $z_{n,j}$, $j\in\{1,2,\dots,n\}$, are real.

In order to determine the location of $z_{n,1}$ and $z_{n,n}$, we use \cite[Theorem~9]{Joulak_2005}. Since
\begin{gather*}f_n(1) f_{n-1}(1)+a_nf_{n-1}(1)+b_n = \frac{\tilde{Q}_{n}(1;\alpha,\beta ,N|q)}{\tilde{Q}_{n-2}(1;\alpha,\beta ,N|q)}+a_n\frac{\tilde{Q}_{n-1}(1;\alpha,\beta ,N|q)}{\tilde{Q}_{n-2}(1;\alpha,\beta ,N|q)}+b_n\\
\hphantom{f_n(1) f_{n-1}(1)+a_nf_{n-1}(1)+b_n}{} = {\frac { \left( \alpha-1 \right) \left( \alpha {q}^{n}-q \right) \left({q}^{n}-q^{N+2} \right) \left({q}^{n}-q^{N+1} \right) q}{ \left( \alpha\beta {q}^{2n} -{q}^{2
} \right) \left( \alpha\beta {q}^{2n} -{q}^{3}\right)
 {q}^{2N}}} < 0,
\end{gather*}
it follows that $z_{n,1}<1$. Furthermore,
\begin{gather*}
f_n\big(q^{-N}\big)f_{n-1}\big(q^{-N}\big)+a_nf_{n-1}\big(q^{-N}\big)+b_n\\
\qquad{} ={\frac { {\alpha}^{2}
 \left( \beta-1 \right) \left( \beta {q}^{n}-q \right) \left(
{q}^{n}-{q}^{N+1} \right) \left({q}^{n} -{q}^{N+2}\right){q}^{n} }{ \left( \alpha\beta {q}^{2n} -{q}^{2
} \right) \left( \alpha\beta {q}^{2n} -{q}^{3}\right){q}^{2N}}}<0
\end{gather*}
and $q^{-N}<z_{n,n}$.
\end{proof}

\begin{Remark}\quad
\begin{itemize}\itemsep=0pt
\item [(i)]
When we let $\b=0$ in the definition of the $q$-Hahn polynomials, we obtain the affine $q$-Krawtchouk polynomials \cite[Section~14.16]{KLS} $\tilde{K}_n^{\text{Aff}}(\bar{x};\a,N;q)$, orthogonal on $\big(1,q^{-N}\big)$ if $0<\a q<1$. When we fix $\a>1$, such that $0<\a q<1$, the quasi-orthogonality of the polynomials $\tilde{K}_n^{\text{Aff}}\big(\bar{x};\frac{\a}{q^k},N;q\big)$, $k<n$, on $\big(1,q^{-N}\big)$ follows directly from~\eqref{eq7qh}, with $\b=0$. If $x_{n,j}$, $j\in\{1,2,\dots,n\}$, denote the zeros of $\tilde{K}_n^{\text{Aff}}(\bar{x};\a,N;q)$ and $y_{n,j}$, $j\in\{1,2,\dots,n\}$, the zeros of $\tilde{K}_n^{\text{Aff}}\big(\bar{x};\frac{\a}{q},N;q\big)$, the interlacing result in Theorem~\ref{Thmqhalpha}(i) follows.
\item [(ii)]
Since $\lim\limits_{\a\rightarrow \infty}\tilde{Q}_n(\overline{x};\alpha,p,N|q)=\tilde{K}_n^{\text{qtm}}(\overline{x};p,N;q)$, \cite[Section~14.14]{KLS}, we obtain from~\eqref{eq2qh}, the equation
\begin{gather*}
\tilde{K}_n^{\text{qtm}}\left(\overline{x};\frac{p}{q},N;q \right)= \tilde{K}_n^{\text{qtm}}(\overline{x};p,N;q) +{\frac { \left( {q}^{N+1}-{q}^{n} \right) \left( {q}^{n}-1 \right) }{p {q}^{2n+N}}\tilde{K}_{n-1}^{\text{qtm}}(\overline{x};p,N;q)}.
\end{gather*}
For $q^{-N}<p<q^{-N+1}$, the quantum $q$-Krawtchouk polynomials $\tilde{K}_n^{\text{qtm}}\big(\overline{x};\frac{p}{q^k},N;q\big)$ are quasi-orthogonal of order $k<n$
and the interlacing result in Theorem~\ref{Thmqhalpha}(ii) follows, where $x_{n,j}$, $j\in\{1,2,\dots,n\}$, denote the zeros of $\tilde{K}_n^{\text{qtm}}(\overline{x};p,N;q)$ and $z_{n,j}$, $j\in\{1,2,\dots,n\}$, the zeros of $\tilde{K}_n^{\text{qtm}}\big(\overline{x};\frac{p}{q},N;q\big)$.
\end{itemize}
\end{Remark}

\subsection[The little $q$-Jacobi polynomials]{The little $\boldsymbol{q}$-Jacobi polynomials}
The little $q$-Jacobi polynomials
 \begin{gather*}\tilde{p}_n(x;\alpha,\beta|q)=(-1)^nq^{\binom{n}{2}}\frac{(\alpha q;q)_n}{(\alpha\beta qq^n;q)_n}\,\qhypergeom{2}{1}{q^{-n},\alpha\beta q^{n+1}}{\alpha q}{q;qx}\end{gather*} are orthogonal with respect to the discrete weight $w(x)=\frac{(\b q;q)_x (\a q)^x}{(q;q)_x}$ for $0<\alpha q<1$, $\beta q<1$ on $(0,1)$. Consider the recurrence equations (cf.~\cite[equations~(9a) and (9b)]{KoepfJT_2017})
\begin{subequations}
 \begin{gather}
\tilde{p}_n\left(x;\frac{\alpha}{q},\beta|q\right)= \tilde{p}_n(x;\alpha ,\beta|q)+\frac {\alpha {q}^{n} \left( {q}^{n}-1 \right) \left( \beta {q}^{n
}-1 \right) }{ \left( \alpha\beta {q}^{2n}-1 \right) \left( \alpha\beta {q}^{2n}-q \right) } \tilde{p}_{n-1}(x;\alpha,\beta|q);\label{eq1lqj} \\
\tilde{p}_n\left(x;\alpha,\frac{\beta}{q}|q \right)=\tilde{p}_n(x;\alpha,\beta |q) -{\frac {\alpha\beta {q}^{2n} \left({q}^{n}-1 \right) \left(\alpha {q}^{n} -1 \right) }{ \left( \alpha\beta {q}^{2n}-1 \right) \left( \alpha\beta {q}^{2n}-q \right) }\tilde{p}_{n-1}(x;\alpha,\beta |q)}.\label{eq2lqj}
 \end{gather}
\end{subequations}
\begin{Corollary}
 \begin{gather}
 \tilde{p}_n\left(x;\frac{\alpha}{q},\frac{\beta}{q} |q \right)=\tilde{p}_{n}(x;\alpha,\beta |q) \nonumber\\
 \hphantom{\tilde{p}_n\left(x;\frac{\alpha}{q},\frac{\beta}{q} |q \right)=}{} -\frac {\alpha {q}^{n} \left( \alpha\beta {q}^{2n}-{q}^{n+1}\beta-\beta {q}^{n}+q \right) \left( {q}^{n}-1 \right)}{ \left( \alpha
\beta {q}^{2n}-{q}^{2} \right) \left( \alpha\beta {q}^{2n}-1 \right) }\tilde{p}_{n-1}(x;\alpha,\beta |q)\nonumber\\
 \hphantom{\tilde{p}_n\left(x;\frac{\alpha}{q},\frac{\beta}{q} |q \right)=}{} -\frac {{\alpha}^{2}\beta{q}^{3 n+1} \left( {q}^{n}-1 \right) \left( \beta {q}^{n}-q \right) \left( \alpha {q}^{n}-q \right) \left( {q}^{n}-q \right)}{ \left( \alpha\beta {q}^{2n}-{q}^{2} \right) ^{2} \left( \alpha\beta {q}^{2n}-q \right) \left( \alpha\beta {q}^{2n}-{q}^{3} \right) }\tilde{p}_{n-2}(x;\alpha,\beta |q).\label{eq8lqj}
 \end{gather}
\end{Corollary}

\begin{Theorem}\label{Thmqolqj}
Let $k,m\in\nn_0$, $\a,\b\in \rr$. For $0<\a q<1$ and $0< \beta q<1$, the sequence of little $q$-Jacobi polynomials
\begin{itemize}\itemsep=0pt
\item[$(i)$] $\big\{\tilde{p}_n\big(x;\frac{\alpha}{q^k},\b |q\big)\big\}_{n=0}^{\infty}$, with $\a>1$, is quasi-orthogonal of order $k\le n-1$ with respect to~$w(x)$ on the interval $(0,1)$ and the polynomials have at least $(n-k)$ real, distinct zeros in~$(0,1)$;
\item[$(ii)$] $\big\{\tilde{p}_n\big(x;\a,\frac{\beta}{q^m} |q\big)\big\}_{n=0}^{\infty}$, $\b>1$, is quasi-orthogonal of order $m \le n-1$ with respect to~$w(x)$ on~$(0,1)$ and the polynomials have at least $(n-m)$ real, distinct zeros in $(0,1)$;
\item[$(iii)$]
$\big\{\tilde{p}_n\big(x;\frac{\alpha}{q^k},\frac{\beta}{q^m} |q\big)\big\}_{n=0}^{\infty}$, $\a,\b>1$, is quasi-orthogonal of order $k+m \le n-1$ with respect to $w(x)$ on $(0,1)$ and the polynomials have at least $n-(k+m)$ real, distinct zeros in $(0,1)$.
\end{itemize}
\end{Theorem}
\begin{proof}(i) Fix $\a>1$, $\b\in\rr$, such that $0<\alpha q<1$, $0<\beta q<1$. From Lemma~\ref{Brez} and~\eqref{eq1lqj}, it follows that $\tilde{p}_n\big(x;\frac{\alpha}{q},\b |q\big)$ is quasi-orthogonal of order one on $(0,1)$. By iteration, we can express $\tilde{p}_n\big(x;\frac{\alpha}{q^k},\b |q\big)$ as a linear combination of $\tilde{p}_n(x;\alpha,\b |q)$, $\tilde{p}_{n-1}(x;\alpha,\b |q)$, $\dots$, $\tilde{p}_{n-k}(x;\alpha,\b |q)$, and the result follows from Lemma~\ref{Brez}. The location of the $(n-k)$ real, distinct zeros of $\tilde{p}_n(x;\frac{\alpha}{q^k},\b |q)$, $k\in\{1,2,\dots,n-1\}$, follows from Lemma~\ref{location}.

(ii) Fix $\b>1$, $\a\in\rr$, such that $0<\alpha q<1$, $0<\beta q<1$. The quasi-orthogonality follows in the same way as in~(i), by using~\eqref{eq2lqj}.

(iii) Fix $\a>1$ and $\b>1$ such that $0<\alpha q<1$ and $0< \beta q<1$. From~\eqref{eq8lqj}, we see that $\tilde{p}_n\big(x;\frac{\alpha}{q},\frac{\beta}{q} |q\big)$ can be written as a linear combination of $\tilde{p}_n(x;\alpha,\b |q)$, $\tilde{p}_{n-1}(x;\alpha,\b |q)$ and $\tilde{p}_{n-2}(x;\alpha,\b |q)$, and it follows from Lemma~\ref{Brez} that the sequence $\tilde{p}_n\big(x;\frac{\alpha}{q},\frac{\beta}{q} |q\big)$ is quasi-orthogonal of order two on $(0,1)$. By iteration, we can express $\tilde{p}_n\big(x;\frac{\alpha}{q^k},\frac{\beta}{q^m} |q\big)$ as a linear combination of $\tilde{p}_n(x;\alpha,\b |q),\tilde{p}_{n-1}(x;\alpha,\b |q),\dots,\tilde{p}_{n-(k+m)}(x;\alpha,\b |q)$, and the results follow directly from Lemmas~\ref{Brez} and~\ref{location}.
\end{proof}

\begin{Theorem}\label{Thmlqjalpha}
Let $\a,\b\in \rr$, $0<\a q,\beta q<1$, and suppose $x_{n,j}$, $j\in\{1,2,\dots,n\}$, denote the zeros of $\tilde{p}_n(x;\alpha,\b |q)$, $y_{n,j}$, $j\in\{1,2,\dots,n\}$, the zeros of $\tilde{p}_n\big(x;\frac{\alpha}{q},\b |q\big)$ and $z_{n,j}$, $j\in\{1,2,\dots,n\}$, the zeros of $\tilde{p}_n\big(x;\a,\frac{\b}{q}|q\big)$. Then
\begin{itemize}\itemsep=0pt
\item[$(i)$] if $\a>1$, $y_{n,1}<0<x_{n,1}<x_{n-1,1}<y_{n,2}<\cdots<x_{n-1,n-1}<y_{n,n}<x_{n,n}<1$;
\item[$(ii)$] if $\b>1$, $0<x_{n,1}<z_{n,1}<x_{n-1,1}<x_{n,2}<\cdots<x_{n-1,n-1}<x_{n,n}<1<z_{n,n}$.
\end{itemize}
\end{Theorem}

\begin{proof}(i) From \eqref{eq1lqj} we obtain the value
\begin{gather*} a_n={\frac { \a\left( {q}^{n}-1 \right) \left( \beta {q}^{n}-1 \right) {q}
^{n}}{ \left( \alpha\beta {q}^{2n}-1 \right) \left( \alpha
 \beta {q}^{2n}-q \right) }}>0.
 \end{gather*} The interlacing result, as well as the position of $y_{n,n}$, follows from Lemma~\ref{Joulak2}(ii).

To obtain the position of $y_{n,1}$ we use Lemma~\ref{Joulak1} and when we consider the given parameter values, \begin{gather*} f_n(0)=\frac{\tilde{p}_n(0;\alpha,\b |q)}{\tilde{p}_{n-1}(0;\alpha,\b |q)}=-{\frac { \left( \alpha\beta {q}^{n}-1 \right) \left( \alpha {q}^
{n}-1 \right) {q}^{n}}{ \left( \alpha\beta {q}^{2n}-1 \right)
 \left( \alpha\beta {q}^{2n}-q \right) }}<0.
 \end{gather*}
We thus have
\begin{gather*} -a_n-f_n(0)={\frac {{q}^{n} ( \alpha-1 ) }{\alpha\beta {q}^{2n}-q}<0
}\end{gather*}
 and the result follows from Lemma~\ref{Joulak1}(i).

(ii) From \eqref{eq2lqj}, we obtain the value
\begin{gather*} a_n=-{\frac {\a\b \left( {q}^{n}-1 \right) \left( \alpha {q}^{n}-1 \right)
{q}^{2n}}{ \left( \alpha\beta {q}^{2n}-1
 \right) \left( \alpha\beta {q}^{2n}-q \right) }}
<0.\end{gather*} The interlacing result, as well as the position of~$z_{n,1}$, follows from Lemma~\ref{Joulak2}(i).

To obtain the position of $z_{n,n}$, we use Lemma~\ref{Joulak1}, and when we consider the given parameter values,
\begin{gather*}
f_n(1)=\frac{\tilde{p}_n(1;\alpha,\b |q)}{\tilde{p}_{n-1}(1;\alpha,\b |q)}=-{\frac {\alpha \left( \beta-1 \right) {q}^{2n}}{\alpha\beta {q }^{2n}-q}}
>0.\end{gather*}
We thus have
\begin{gather*} -a_n-f_n(1)=-{\frac {\alpha \left( \beta-1 \right) {q}^{2n}}{\alpha\beta {q
}^{2n}-q}}>0\end{gather*}
 and it follows from Lemma \ref{Joulak1}(ii) that $1<z_{n,n}$.
\end{proof}

\begin{Theorem}\label{Thmqh4+}
Let $\a,\b>1$. All the zeros of $\tilde{p}_n\big(x;\frac{\a}{q},\frac{\b}{q} |q\big)$, denoted by $z_{n,j}$, $j\in\{1,2,\dots,n\}$, are real and distinct and $z_{n,1}<0$ and $1<z_{n,n}$.
 \end{Theorem}
\begin{proof} Fix $\a>1$ and $\b>1$ such that $0<\alpha q<1$ and $0< \beta q<1$. We use~\eqref{eq8lqj}, with $a_n$ the coefficient of $\tilde{p}_{n-1}(x;\alpha,\b |q)$ and $b_n$ the coefficient of $\tilde{p}_{n-2}(x;\alpha,\b |q)$. By taking into account the values of the parameters, we see that
\begin{gather*}
b_n= -{\frac {{\alpha}^{2} \b \left( \beta {q}^{n}-q \right) \left( \alpha
{q}^{n}-q \right) \left( {q}^{n}-1 \right) \left( {q}^{n
}-q \right){q}^{3 n+1}}{ \left( \alpha\beta {q}^{2n}-q \right)
 \left( \alpha\beta {q}^{2n}-{q}^{3} \right) \left( \alpha
\beta {q}^{2n}-{q}^{2} \right) ^{2}}}<0,
\end{gather*}
and it follows from \cite[Theorem~4]{BDR} that $z_{n,j}$, $j\in\{1,2,\dots,n\}$, are real and distinct.

In order to determine the location of $z_{n,1}$ and $z_{n,n}$, we use \cite[Theorem~9]{Joulak_2005}. Since
\begin{gather*}
f_n(0) f_{n-1}(0)+a_nf_{n-1}(0)+b_n={\frac { \left( \alpha-1 \right) \left(
\alpha {q}^{n}-q \right) q^{2n+1}}{ \left( \alpha\beta {q}^{2n}
 -{q}^{3} \right) \left( \alpha\beta {q}^{2n}-{q}^{2} \right) }}<0,
\end{gather*}
 it follows that $z_{n,1}<1$. Furthermore,
\begin{gather*}
f_n(1)f_{n-1}(1)+a_nf_{n-1}(1)+b_n
={\frac { {\alpha}^{2} \left( \beta-1 \right) \left( \beta {
q}^{n}-q \right){q}^{3 n} }{ \left( \alpha\beta {q}^{2n}-{q}^{2} \right)
 \left( \alpha\beta {q}^{2n}-{q}^{3} \right) }}<0,
\end{gather*}
and $1<z_{n,n}$.
\end{proof}

\begin{Remark}\quad
\begin{itemize}\itemsep=0pt
\item[(i)]
From \eqref{eq1lqj} we obtain
\begin{gather*}
\tilde{p}_n\left(x;\frac{\alpha}{q^2},\beta|q\right)= \tilde{p}_{n}(x;\alpha,\beta |q)+\frac { \alpha {q}^{n}\left( q+1 \right) \left( {q}^{n}-1 \right) \left( \beta {q}^{n}-1 \right)}{ \left( \alpha\beta {q}^{2n}-{q}^{2} \right) \left( \alpha\beta {q}^{2n}-1 \right) }\\
\hphantom{\tilde{p}_n\left(x;\frac{\alpha}{q^2},\beta|q\right)=}{}\times \tilde{p}_{n-1}(x;\alpha,\beta |q)
+b_n \tilde{p}_{n-2}(x;\alpha,\beta |q),
\end{gather*}
with
\begin{gather*}
b_n=\frac { {\alpha}^{2} {q}^{2n+2} \left( {q}^{n}-1 \right) \left( \beta {q}^{n}-q \right) \left( {q}^{n}-q \right)\left( \beta {q}^{n}-1 \right) }{ \left( \alpha\beta {q}^{2n}-{q}^{2} \right) ^{2} \left( \alpha\beta {q}^{2n}-q \right) \left( \alpha\beta {q}^{2n}-{q}^{3} \right) }
\end{gather*}
and
\begin{gather*} C_n-b_n={\frac {{q}^{2n+1} \left( \alpha-q \right) \left( {q}^{n}-q
 \right) \left( \b{q}^{n}-q \right) \alpha}{ \left(\a \b {q}^{2n}
-{q}^{3} \right) \left(\a \b {q}^{2n}-{q}^{2}
 \right) ^{2}}},
 \end{gather*} where $-C_n$ is the coefficient of $\tilde{p}_{n-2}(x;\alpha,\b |q)$ in the three-term recurrence equation of the little $q$-Jacobi polynomials \cite[equation~(14.12.4)]{KLS}.
Since $C_n<b_n$, there is an interlacing between $(n-2)$ zeros of $\tilde{p}_n\big(x;\frac{\alpha}{q^2},\beta|q\big)$ and the $(n-1)$ zeros of $\tilde{p}_{n-1}(x;\alpha,\b |q)$ (cf.\ \cite[Theorem~15]{Joulak_2005}).
\item[(ii)]
 When $\b=0$ in the definition of the little $q$-Jacobi polynomials, we obtain the little $q$-Laguerre (or Wall) polynomials $\tilde{p}_n(x;\a|q) $, that are orthogonal on $(0,1)$ when $0<\a q<1$. The quasi-orthogonality of $\big\{\tilde{p}_n\big(x;\frac{\a}{q^k}|q\big)\big\}_{n\ge0}$, for $k<n$, when $\a>1$, $0<\a q<1$, follows directly from \eqref{eq1lqj} (with $\b=0$). The location of the zeros of the order one quasi-orthogonal polynomial $\tilde{p}_n\big(x;\frac{\a}{q}|q\big)$ is given in Theorem~\ref{Thmlqjalpha}(i), where
$x_{n,j}$, $j\in\{1,2,\dots,n\}$, denote the zeros of $\tilde{p}_n(x;\alpha |q)$ and $y_{n,j}$, $j\in\{1,2,\dots,n\}$, the zeros of $\tilde{p}_{n}\big(x;\frac{\alpha}{q}|q\big)$.
\end{itemize}
\end{Remark}

\subsection[The $q$-Laguerre polynomials]{The $\boldsymbol{q}$-Laguerre polynomials}
The $q$-Laguerre polynomials
 \begin{gather*}\tilde{L}_n^{(\alpha)}(x;q)={(-1)^n(q^{\alpha+1};q)_n\over q^{n(n+\alpha)}}\,\qhypergeom{1}{1}{q^{-n}}{q^{\alpha+1}}{q;-q^{n+\alpha+1}x}\end{gather*} are orthogonal for $\alpha>-1$ on $(0,\infty)$ with respect to the weight function $w(x)=\frac{x^{\a}}{(-x;q)_{\infty}}$. Consider the equation (cf.\ \cite[equation~(4.12)]{Moak} and \cite[equation~(12a)]{KoepfJT_2017})
\begin{gather}\label{eq1ql}
 \tilde{L}_n^{(\alpha-1)}(x;q)= \tilde{L}_n^{(\alpha)}(x;q) -\frac{\left( {q}^{n}-1 \right)q}{{q}^{2n+\alpha}} \tilde{L}_{n-1}^{(\alpha)}(x;q).
\end{gather}

\begin{Theorem}\label{Thmqolqj+}
Let $k\in\nn_0$ and $\a\in \rr$. For $-1<\a<0$ and $k\in\{1,2,\dots,n-1\}$, the sequence of $q$-Laguerre polynomials $\big\{\tilde{L}_n^{(\alpha-k)}(x;q)\big\}_{n=0}^{\infty}$ is quasi-orthogonal of order~$k$ on the interval $(0,\infty)$ with respect to~$w(x)$ and the polynomials have at least~$(n-k)$ real, distinct zeros in $(0,\infty)$.
\end{Theorem}
\begin{proof} Fix $-1<\a<0$. From Lemma~\ref{Brez} and~\eqref{eq1ql} it follows that $\tilde{L}_n^{(\alpha-1)}(x;q)$ is quasi-orthogonal of order one on~$(0,\infty)$. By iteration, we can express $\tilde{L}_n^{(\alpha-k)}(x;q)$ as a linear combination of $\tilde{L}_{n}^{(\alpha)}(x;q),\tilde{L}_{n-1}^{(\alpha)}(x;q),\dots,\tilde{L}_{n-k}^{(\alpha)}(x;q)$, and the result follows from Lemma~\ref{Brez}. The location of the $(n-k)$ real, distinct zeros of $\tilde{L}_n^{(\alpha-k)}(x;q)$, $k\in\{1,2,\dots,n-1\}$, follows from Lemma \ref{location}.
\end{proof}

\begin{Theorem}\label{Thmlql}
Let $-1<\a<0$ and denote the zeros of $\tilde{L}_n^{(\alpha)}(x;q)$ by $x_{n,j}$, $j\in\{1,2,\dots,n\}$, and the zeros of $\tilde{L}_n^{(\alpha-1)}(x;q)$ by $y_{n,j}$, $j\in\{1,2,\dots,n\}$. Then
\begin{gather*} y_{n,1}<0<x_{n,1}<x_{n-1,1}<y_{n,2}<x_{n,2}<\cdots<x_{n-1,n-1}<y_{n,n}<x_{n,n}.\end{gather*}
\end{Theorem}
\begin{proof}
From \eqref{eq1ql}, we obtain the value $a_n={\frac { - \left( {q}^{n}-1 \right) {q}}{q^{2n+\alpha} }}>0$. The interlacing result, as well as the position of $y_{n,n}$, follows from Lemma~\ref{Joulak2}(ii).

To obtain the position of $y_{n,1}$, we use Lemma~\ref{Joulak1}, and when we consider the given parameter values,
\begin{gather*} f_n(0)= \frac{\left( {q}^{n+\alpha}-1 \right)q}{ {q}^{2n+\alpha}}<0.\end{gather*}
We thus have
\begin{gather*} -a_n-f_n(0)=-\frac{\left( {q}^{\alpha}-1 \right)q}{ {q}^{n+\alpha}}<0\end{gather*}
 and since $-a_n<f_n(0)<0$, the result follows from Lemma~\ref{Joulak1}(i).
\end{proof}

\subsection{The Al-Salam--Carlitz I polynomials}
The Al-Salam--Carlitz I polynomials
 \begin{gather*}\tilde{U}_n^{(\alpha)}(x;q)=(-\alpha)^nq^{\binom{n}{2}}\,\qhypergeom{2}{1}{q^{-n},x^{-1}}{0}{q;{qx\over \alpha}}\end{gather*}
are orthogonal for $\a<0$ on $(\a,1)$ with respect to the weight function $w(x)=\big(qx,\frac{qx}{\a};q\big)_{\infty}$. The polynomials $\tilde{U}_n^{(\frac{\a}{q^k})}(x;q)$, $k<n$, are orthogonal with respect to $\big(qx,\frac{q^{k+1}x}{\a};q\big)_{\infty}$ on the interval $\big(\frac{\a}{q^k},1\big)$ and we will prove that they are quasi-orthogonal with respect to $w(x)$ on $(\a,1)$. We use the equation
\begin{gather}
\tilde{U}_n^{(\frac{\a}{q})}(x;q)=\tilde{U}_n^{(\alpha )}(x;q)+\alpha q^{-1}\big(q^n-1\big)\tilde{U}_{n-1}^{(\alpha )}(x;q).\label{eq1asc1}
\end{gather}
We deduce that
\begin{gather}
 \tilde{U}_n^{(\frac{\a}{q^2})}(x;q) =\tilde{U}_n^{(\alpha )}(x;q) +\frac{\alpha \left( {q}^{n}-1 \right) \left( q+1 \right)}{{q}^{2}}\tilde{U}_{n-1}^{(\alpha )}(x;q)\nonumber\\
 \hphantom{\tilde{U}_n^{(\frac{\a}{q^2})}(x;q) =}{} +\frac{{\alpha}^{2} \left( {q}^{n}-1 \right) \left( {q}^{n}-q \right)}{{q}^{4}} \tilde{U}_{n-2}^{(\alpha )}(x;q).\label{eq1asc2}
\end{gather}

\begin{Theorem}\label{Thmqoalsalc1}
Let $k\in\nn_0$ and $\a<0$. The sequence of Al-Salam--Carlitz~I polynomials $\big\{\tilde{U}_n^{(\frac{\a}{q^k})}(x;q)\big\}_{n=0}^{\infty}$ is quasi-orthogonal with respect to $w(x)$ on $(\a,1)$ and the polynomials have at least $(n-k)$ real, distinct zeros in~$(\a,1)$.
\end{Theorem}
\begin{proof} From Lemma~\ref{Brez} and~\eqref{eq1asc1} it follows that $\tilde{U}_n^{(\frac{\a}{q})}(x;q)$ is quasi-orthogonal of order one on $(\a,1)$. By iteration, we can express $\tilde{U}_n^{(\frac{\a}{q^k})}(x;q)$ as a linear combination of $\tilde{U}_{n}^{(\alpha )}(x;q)$, $\tilde{U}_{n-1}^{(\alpha )}(x;q)$, $\dots,\tilde{U}_{n-k}^{(\alpha )}(x;q)$, and the result follows from Lemma~\ref{Brez}. The location of the $(n-k)$ real, distinct zeros of $\tilde{U}_n^{(\frac{\a}{q^k})}(x;q)$, $k\in\{1,2,\dots,n-1\}$, follows from Lemma~\ref{location}.
\end{proof}

\begin{Remark} We can also obtain \eqref{eq1asc1} from the generating function \cite[equation~(14.24.10)]{KLS} of the Al-Salam--Carlitz~I polynomials
\begin{gather}\label{gen_ASC1}
 \frac{\left(t,\alpha t;q\right)_\infty}{(xt;q)_\infty}=\sum_{n=0}^{\infty}\frac{U_n^{(\alpha)}(x;q)}{(q;q)_n}t^n,
\end{gather}
from which it follows that
\begin{gather}\label{ASC1_1}
 \frac{(t,\frac{\alpha}{q^k} t;q)_\infty}{(xt;q)_\infty}=\sum_{n=0}^{\infty} \frac{U_n^{(\frac{\alpha}{q^k})}(x;q)}{(q;q)_n}t^n,\qquad k\in\{1,2,\ldots\}.
\end{gather}
From the relation
\begin{gather*}\frac{(a;q)_\infty}{(a q^n;q)_\infty}=(a;q)_n,\end{gather*}
we obtain, when $a=\frac{\alpha}{q^k}t$ and $n=k$,
\begin{gather*}\left(\frac{\alpha}{q^k}t;q\right)_\infty=\left(\frac{\alpha}{q^k}t;q\right)_k(\alpha t;q)_\infty.\end{gather*}
By using \eqref{gen_ASC1}, \eqref{ASC1_1} becomes
\begin{gather*}
 \left(\frac{\alpha}{q^k}t;q\right)_k\sum_{n=0}^{\infty} \frac{U_n^{(\alpha)}(x;q)}{(q;q)_n}t^n=\sum_{n=0}^{\infty} \frac{U_n^{(\frac{\alpha}{q^k}t)}(x;q)}{(q;q)_n}t^n,\qquad k\in\{1,2,\ldots\}.
\end{gather*}
Expanding $\big(\frac{\alpha}{q^k}t;q\big)_k$ and equating powers of~$t$ yields $U_n^{(\frac{\alpha}{q^k})}(x;q)$ as a linear combination of $U_{n-j}^{(\alpha)}(x;q)$, $j\in\{0,1,\ldots,k\}$. In particular, for $k=1$ and $k=2$, we get~\eqref{eq1asc1} and~\eqref{eq1asc2}, respectively.
 \end{Remark}

\begin{Theorem}\label{Thmlql+}
Let $\a<0$ and denote the zeros of $\tilde{U}_{n}^{(\alpha )}(x;q)$ by $x_{n,j}$, $j\in\{1,2,\dots,n\}$, and the zeros of $\tilde{U}_{n}^{(\frac{\a}{q})}(x;q)$ by $y_{n,j}$, $j\in\{1,2,\dots,n\}$. Then
\begin{itemize}\itemsep=0pt
\item[$(i)$] $\frac{\a}{q}<y_{n,1}<x_{n,1}<x_{n-1,1}<y_{n,2}<\cdots<x_{n-1,n-1}<y_{n,n}<x_{n,n}<1$
 and, additionally, if $\a<\frac{q^n}{q^n-1}$, then $y_{n,1}<\a<x_{n,1}$;
\item[$(ii)$] $(n-2)$ zeros of $\tilde{U}_{n}^{(\frac{\a}{q^2})}(x;q)$ interlace with the $n$ zeros of $\tilde{U}_{n}^{({\a})}(x;q)$ if $\a<\frac{q^{n+1}}{q^n-1}$.
\end{itemize}
\end{Theorem}
\begin{proof}(i) From \eqref{eq1asc1} we obtain the value $a_n=\frac{\alpha (q^n-1)}{q}>0$. The interlacing result, as well as the position of $y_{n,n}$, follows from Lemma~\ref{Joulak2}(ii). The position of $y_{n,1}$ cannot be determined, since
\begin{gather*} f_n(\a)= \frac{\tilde{U}_{n}^{(\alpha )}(\a;q)}{\tilde{U}_{n-1}^{(\alpha )}(\a;q)}=-{q}^{n-1}<0
\end{gather*}
and the sign of
\begin{gather*} -a_n-f_n(\a)=-\frac{\alpha (q^n-1)}{q}+{q}^{n-1}=\frac{\alpha (1-q^n)+{q}^{n}}{q}
\end{gather*}
varies as the parameters vary within the allowed regions. However, if $\a<\frac{q^n}{q^n-1}$, then $-a_n<f_n(\a)<0$ and from Lemma~\ref{Joulak1}(i), it follows that $y_{n,1}<\a$.

(ii) The coefficient of $ \tilde{U}_{n-2}^{(\alpha )}(x;q)$ in~\eqref{eq1asc2} is
\begin{gather*} b_n={\frac { \left( {q}^{n}-1 \right) \left( {q}^{n}-q \right) {\alpha}^{
2}}{{q}^{4}}}.
\end{gather*}
From the three-term recurrence equation of the Al-Salam--Carlitz I polynomials \cite[equation~(14.24.4)]{KLS}
\begin{gather*} \tilde{U}_{n}^{(\alpha )}(x;q) = \left( x-{\frac {{q}^{n} \left( \alpha+1 \right)
}{q}} \right) \tilde{U}_{n-1}^{(\alpha )}(x;q) -{\frac {\alpha {q}^{n} \left( {q}
^{n}-q \right)}{{q}^{3}}}\tilde{U}_{n-2}^{(\alpha )}(x;q),
\end{gather*} we obtain $C_n= {\frac {\alpha {q}^{n} \left( {q}
^{n}-q \right)}{{q}^{3}}}$ and since
\begin{gather*} b_n-C_n={\frac {\alpha \left( {q}^{n}-q \right) \left( \alpha {q}^{n}-{q}^
{n+1}-\alpha \right) }{{q}^{4}}}>0
\end{gather*} when $\a<\frac{q^{n+1}}{q^n-1}$, the interlacing result follows from \cite[Theorem~15(ii)]{Joulak_2005}.
\end{proof}

\section[Classical orthogonal polynomials on a $q$-quadratic lattice]{Classical orthogonal polynomials on a $\boldsymbol{q}$-quadratic lattice}\label{section3}

In this section we consider the quasi-orthogonality of the monic Askey--Wilson and $q$-Racah polynomials, that are defined on a $q$-quadratic lattice.
\subsection{The Askey--Wilson polynomials}
The Askey--Wilson polynomials
\begin{gather*}
 \tilde{p}_n(x;a,b,c,d|q)={(ab,ac,ad;q)_n\over (2a)^n (abcdq^{n-1};q)_n}\, \qhypergeom{4}{3}{q^{-n},abcdq^{n-1},ae^{i\theta},ae^{-i\theta}}{ab,ac,ad}{q;q},\quad x=\cos\theta,
 \end{gather*}
 with $a$, $b$, $c$, $d$ either real, or they occur in complex conjugate pairs, and $\max(|a|,|b|,|c|,|d|)<1$, are orthogonal on $(-1,1)$ with respect to
 \begin{gather}\label{weightAW}w(x;a,b,c,d|q)=\frac{1}{\sqrt{1-x^2}}\left|\frac{(e^{2 i \theta};q)_{\infty}}{(a e^{i \theta},b e^{i \theta},c e^{i \theta},d e^{i \theta};q)_{\infty}}\right|^2.\end{gather}
The weight function is independent of the order of the parameters $a$, $b$, $c$ and $d$ and by shifting~$b$ to~$b/q$, $c$ to $c/q$ or $d$ to $d/q$, we obtain the same interlacing results as by shifting~$a$ to~$a/q$.
We will now fix $a>0$ such that $q<|a|<1$ and for these values of~$a$, the polynomial $\tilde{p}_n(x;a ,b,c,d|q)$ is orthogonal on $(-1,1)$ with respect to $w(x;a,b,c,d|q)$. In what follows, we assume that $|a|=\max(|a|,|b|,|c|,|d|)<1$. Should this not be the case, the order in which the parameters occur, can be changed.

We will thus only consider the equations in which $a$ is shifted to $\frac {a}{q^k}>1$ (or $\frac {a}{q^k}<-1$ should $a<0$), and we will prove that the polynomials $\tilde{p}_n\big(x;\frac{a}{q^k},b,c,d|q\big)$, $k\in\{1,2,\dots,n-1\}$, are quasi-orthogonal of order~$k$ on~$(-1,1)$. We use the equation \cite[equation~(17a)]{KoepfJT_2017}
\begin{gather}
\tilde{p}_n\left(x;\frac{a}{q},b,c,d|q\right) = \tilde{p}_n(x;a,b,c,d|q)\nonumber\\
\hphantom{\tilde{p}_n\left(x;\frac{a}{q},b,c,d|q\right) =}{} -{\frac {aq \left( {q}^{n}-1 \right) \left( cd{q}^{n}-q \right)
 \left(bd {q}^{n}-q \right) \left(bc {q}^{n}-q \right)}{2 \left ( abcd{q}^{2n}-{q}^{3} \right) \left(abcd{q
}^{2n}-{q}^{2} \right)}}\tilde{p}_{n-1}(x;a,b,c,d|q).\label{eqaw1}\!\!\!
\end{gather}
\begin{Theorem}\label{ThmAW}
Let $a$, $b$, $c$, $d$ be real, or they occur in complex conjugate pairs if complex, and $\max(|a|,|b|,|c|,|d|)<1$, and let $w(x;a,b,c,d|q)$ be as defined in~\eqref{weightAW}. For $a$ such that \smash{$q<|a|<1$}, the sequence of Askey--Wilson polynomials $\big\{\tilde{p}_n\big(x;\frac{a}{q^k},b,c,d|q\big)\big\}_{n=0}^{\infty}$ is quasi-ortho\-go\-nal of order $k<n$ with respect to the weight $w(x;a,b,c,d|q)$ on the interval $(-1,1)$ and the polynomials have at least $(n-k)$ real, distinct zeros in~$(-1,1)$.
\end{Theorem}
\begin{proof}Suppose $q<|a|<1$. From Lemma \ref{Brez} and~\eqref{eqaw1}, it follows that $\tilde{p}_n\big(x;\frac{a}{q},b,c,d|q\big)$ is quasi-orthogonal of order one on $(-1,1)$. By iteration, we can express $\tilde{p}_n\big(x;\frac{a}{q^k},b,c,d|q\big)$ as a~linear combination of $\tilde{p}_n(x;{a},b,c,d|q),\tilde{p}_{n-1}(x;{a},b,c,d|q),\dots,\tilde{p}_{n-k}(x;{a},b,c,d|q)$ and the result follows from Lemma~\ref{Brez}. The location of the $(n-k)$ real, distinct zeros of $\tilde{p}_n\big(x;\frac{a}{q^k},b,c,d|q\big)$, $k\in\{1,2,\dots,n-1\}$, follows from Lemma~\ref{location}.
\end{proof}

\begin{Theorem}\label{Thmawlocation}Let $a$, $b$, $c$, $d$ be real, or they occur in complex conjugate pairs if complex. Suppose $|a|=\max(|a|,|b|,|c|,|d|)<1$, $q<|a|<1$ and let $x_{n,i}$, $i\in\{1,2,\dots,n\}$, denote the zeros of $\tilde{p}_n(x;{a},b,c,d|q)$ and $y_{n,i}$, $i\in\{1,2,\dots,n\}$, the zeros of $\tilde{p}_n\big(x;\frac{a}{q},b,c,d|q\big)$. Then
\begin{itemize}\itemsep=0pt
\item[$(i)$] if $a>0$, $-1<x_{n,1}<y_{n,1}<x_{n-1,1}<x_{n,2}<y_{n,2}<\cdots<x_{n-1,n-1}<x_{n,n}<y_{n,n}$;
\item[$(ii)$] if $a<0$, $y_{n,1}<x_{n,1}<x_{n-1,1}<y_{n,2}<x_{n,2}<\cdots<x_{n-1,n-1}<y_{n,n}<x_{n,n}<1$.
\end{itemize}
\end{Theorem}
\begin{proof}
Suppose $|a|=\max(|a|,|b|,|c|,|d|)<1$. The coefficient of $\tilde{p}_{n-1}(x;{a},b,c,d|q)$ in \eqref{eqaw1} is
\begin{gather*}a_n={-\frac {aq \left( {q}^{n}-1 \right) \left( cd{q}^{n}-q \right)
 \left(bd {q}^{n}-q \right) \left(bc {q}^{n}-q \right)}{2 \left ( abcd{q}^{2n}-{q}^{3} \right) \left(abcd{q
}^{2n}-{q}^{2} \right)}}.\end{gather*}

(i) Consider the case $a>0$ and fix $a$ such that $q<a<1$. Then $a_n<0$ for the given parameter values and the interlacing result, as well as the position of $y_{n,1}$, follows from Lemma~\ref{Joulak2}(i).

(ii) Now we consider the case $a<0$ and fix $a$ such that $-1<a<-q$. Then $a_n>0$ and the interlacing result, as well as the position of $y_{n,n}$, follows from Lemma~\ref{Joulak2}(ii).
\end{proof}

\subsection[The $q$-Racah polynomials]{The $\boldsymbol{q}$-Racah polynomials}
The $q$-Racah polynomials
\begin{gather*}
\tilde{R}_n(\mu(x);\alpha,\beta,\gamma,\delta|q)= \frac{(\alpha q, \beta\delta q,\gamma q;q)_n}{(\alpha\beta q^{n+1};q)_n}\qhypergeom{4}{3}{q^{-n},\alpha\beta q^{n+1},q^{-x},\gamma\delta q^{x+1}}{\alpha q,\beta\delta q,\gamma q}{q;q},\end{gather*}
with $\mu(x)=q^{-x}+\gamma\delta q^{x+1}$, are orthogonal for $n\in\{0,1,\ldots,N\}$, with respect to the discrete weight function
\begin{gather*} w(x)=\frac{(\a q,\b \delta q, \gamma q, \gamma \delta q;q)_x (1-\gamma \delta q^{2x+1})}{\big(q,\frac{\gamma \delta q}{\a},\frac{\gamma q}{\b},
\delta q;q\big)_x(\a \b q)^x(1-\gamma \delta q)}\end{gather*} for
 $\alpha q=q^{-N} $ or $\beta\delta q=q^{-N}$ or $\gamma q=q^{-N}$, where $N$ is a nonnegative integer. Shifting the parameter $\g$ or $\delta$ will change $\mu(x)$ and we will only consider shifts of $\a$ and $\b$. From \cite[equations~(18a) and (18b)]{KoepfJT_2017} we obtain
\begin{subequations}
\begin{gather}
\tilde{R}_{n}\left(\mu(x);\frac{\alpha}{q},\beta,\gamma,\delta|q\right) =\tilde{R}_{n}(\mu(x);\alpha ,\beta,\gamma,\delta|q)\nonumber\\
\qquad{}
 -\frac{\alpha q(1-q^n)(1-\beta q^n)(1-\gamma q^n)(1-\beta\delta q^n)}{(1-\alpha\beta q^{2n})(q-\alpha\beta q^{2n})}\tilde{R}_{n-1}(\mu(x);\alpha ,\beta,\gamma,\delta|q);\label{eqqR1}\\
\tilde{R}_{n}\left(\mu(x);\alpha,\frac{\beta}{q},\gamma,\delta|q\right) =\tilde{R}_{n}(\mu(x);\alpha,\beta,\gamma,\delta|q)\nonumber\\
\qquad{} + \frac{\beta q(1-q^n)(1-\alpha q^n)(1-\gamma q^n)(\alpha q^n-\delta)}{(1-\alpha\beta q^{2n})(q-\alpha\beta q^{2n})}\tilde{R}_{n-1}(\mu(x);\alpha,\beta ,\gamma,\delta|q).\label{eqqR2}
\end{gather}
\end{subequations}

\begin{Theorem}\label{ThmW} Let $k\in\{1,2,\dots,n-1\}$. The sequence of $q$-Racah polynomials
\begin{itemize}\itemsep=0pt
\item[$(i)$] $\big\{\tilde{R}_{n}\big(\mu(x);\frac{\alpha}{q^k},\beta,\gamma,\delta|q\big)\big\}_{n=0}^{N}$, with $\a=q^{-N-1}$, is quasi-orthogonal of order $k$ with respect to the weight $w(x)$ on $(\mu(0),\mu(N))$ and the polynomials have at least $(n-k)$ real, distinct zeros on $(\mu(0),\mu(N))$;
\item[$(ii)$] $\big\{\tilde{R}_{n}\big(\mu(x);\alpha,\frac{\beta}{q},\gamma,\delta|q\big)\big\}_{n=0}^{N}$, with $\b=\frac{q^{-N-1}}{\delta}$, is quasi-orthogonal of order $k$ with respect to the weight $w(x)$ on $(\mu(0),\mu(N))$ and the polynomials have at least $(n-k)$ real, distinct zeros on $(\mu(0),\mu(N))$.
\end{itemize}
\end{Theorem}
\begin{proof}(i) Let $\a=q^{-N-1}$. From Lemma \ref{Brez} and \eqref{eqqR1}, it follows that $\tilde{R}_{n}\big(\mu(x);\frac{\alpha}{q},\beta,\gamma,\delta|q\big)$ is quasi-orthogonal of order one on $(\mu(0),\mu(N))$. By iteration, we can express $\tilde{R}_{n}\big(\mu(x);\frac{\alpha}{q^k},\beta,\gamma,\delta|q\big)\!$ as a linear combination of $\tilde{R}_{n}(\mu(x);\alpha,\beta,\gamma,\delta|q)$, $\tilde{R}_{n-1}(\mu(x);\a,\beta,\gamma,\delta|q)$, $\dots$, $\tilde{R}_{n-k}(\mu(x);\a,\beta$, $\gamma,\delta|q)$ and the result follows from Lemma~\ref{Brez}. Furthermore, from Lemma~\ref{location} we know that at least $(n-k)$ real, distinct zeros of $\tilde{R}_{n}(\mu(x);\frac{\alpha}{q^k},\beta,\gamma,\delta)$, $k\in\{1,2,\dots,n-1\}$, lie in $(\mu(0),\mu(N))$.

(ii) Let $\b=\frac{q^{-N-1}}{\delta}$. The result follow in the same way from \eqref{eqqR2} and Lemmas \ref{Brez} \linebreak and~\ref{location}.
\end{proof}

 For values of $n$ larger than $\frac{N}{2}+1$, we obtain the following interlacing results.
 \begin{Theorem} Consider $n\le N$ and let $x_{n,i}$, $i\in\{1,2,\dots,n\}$, denote the zeros of $\tilde{R}_n(\mu(x);\alpha,\beta$, $\gamma,\delta|q)$, $y_{n,i},i\in\{1,2,\dots,n\}$, the zeros of $\tilde{R}_n\big(\mu(x);\frac{\alpha}{q},\beta, \gamma,\delta|q\big)$ and $z_{n,i}$, $i\in\{1,2,\dots,n\}$, the zeros of $\tilde{R}_n\big(\mu(x);\alpha,\frac{\beta}{q},\gamma,\delta|q\big)$. Then, for $n>\frac{N}{2}+1$,
\begin{itemize}\itemsep=0pt
\item[$(i)$] if $\alpha=q^{-N-1}$ and $\beta q<1$, $\gamma q<1$, $\beta\delta q<1$, \begin{gather*} \mu(0)<x_{n,1}<y_{n,1}<x_{n-1,1}<x_{n,2}<y_{n,2}<\cdots<x_{n-1,n-1}<x_{n,n}<y_{n,n};\end{gather*}
\item[$(ii)$] if $\beta =\frac{q^{-N-1}}{\delta}$ and $\alpha q<1$, $\gamma q<1$, $\frac{\alpha}{\delta}q<1$, we have
\begin{gather*} \mu(0)<x_{n,1}<z_{n,1}<x_{n-1,1}<x_{n,2}<z_{n,2}<\cdots<x_{n-1,n-1}<x_{n,n}<z_{n,n}.\end{gather*}
\end{itemize}
\end{Theorem}
 \begin{proof} Under the above hypotheses, the coefficients of $\tilde{R}_{n-1}(\mu(x);\alpha ,\beta,\gamma,\delta|q)$ in~\eqref{eqqR1} and \eqref{eqqR2} are negative and the interlacing results follow from Lemma~\ref{Joulak2}(i).
 \end{proof}

\section[Quasi-orthogonality of polynomials on a linear and quadratic lattice]{Quasi-orthogonality of polynomials\\ on a linear and quadratic lattice}\label{section4}

In this section we consider the quasi-orthogonality of the monic Wilson and Racah polynomials, that are defined on a quadratic lattice. The quasi-orthogonality of the dual Hahn and continuous dual Hahn polynomials, that also fall in this category, was discussed in~\cite{JJJ}. We also prove the quasi-orthogonality of the monic continuous Hahn polynomials that are defined on a linear lattice.
\subsection{The Wilson polynomials}
The Wilson polynomials
\begin{gather*}
 \tilde{W}_n\big(x^2;a,b,c,d\big)=\frac{(-1)^n(a+b,a+c,a+d)_n}{(n+a+b+c+d-1)_n}\\
 \hphantom{\tilde{W}_n\big(x^2;a,b,c,d\big)=}{}\times \hypergeom{4}{3}{-n,n+a+b+c+d-1,a+ix,a-ix}{a+b,a+c,a+d}{1},
\end{gather*}
are orthogonal on $(0,\infty)$ with respect to
 \begin{gather}\label{weightW}w(x;a,b,c,d)=\left|\frac{\Gamma(a+ix)\Gamma(b+ix)\Gamma(c+ix)\Gamma(d+ix)}{\Gamma(2ix)}\right|^2,\end{gather}
for $\operatorname{Re}(a,b,c,d)>0$ and non-real parameters occur in conjugate pairs. Furthermore, as in the case of the Askey--Wilson polynomials, the weight function is clearly independent of the order in which the parameter $a$, $b$, $c$ and~$d$ occur. We note that the polynomial $W_n\big(x^2;a,b,c,d\big)$ has~$n$ zeros in $(0,\infty)$, namely $(x_{n,1})^2, (x_{n,2})^2,\dots,(x_{n,n})^2$. Let $\tilde{W}_n\big(x^2\big)=\tilde{W}_n\big(x^2;a,b,c,d\big)$.

\begin{Proposition}
\begin{subequations}
\begin{gather}
 \tilde{W}_n\big(x^2;a-1,b,c,d\big)\nonumber\\
 \qquad{} =\tilde{W}_n\big(x^2\big) +{\frac {n( c+d+n-1)( b+d+n-1)( b+c+n-1) }{( 2n+a+b+c+d-2)( 2n+a+b+c+d-3) }}\tilde{W}_{n-1}\big(x^2\big);\label{eqw1}\\
 \tilde{W}_n\big(x^2;a,b-1,c,d\big)\nonumber\\
 \qquad{} =\tilde{W}_n\big(x^2\big)+{\frac {n( c+d+n-1)( a+d+n-1)( a+c+n-1) }{( 2n+a+b+c+d-2)( 2n+a+b+c+d-3) }}\tilde{W}_{n-1}\big(x^2\big);\label{eqw2}\\
 \tilde{W}_n\big(x^2;a,b,c-1,d\big)\nonumber\\
 \qquad{} =\tilde{W}_n\big(x^2\big)+{\frac {n( a+d+n-1)( b+d+n-1)( a+b+n-1) }{( 2n+a+b+c+d-2)( 2n+a+b+c+d-3) }}\tilde{W}_{n-1}\big(x^2\big);\label{eqw3}\\
 \tilde{W}_n\big(x^2;a,b,c,d-1\big)\nonumber\\
 \qquad{} =\tilde{W}_n\big(x^2\big)+{\frac {n( a+c+n-1)( b+c+n-1)( a+b+n-1) }{( 2n+a+b+c+d-2)( 2n+a+b+c+d-3) }}\tilde{W}_{n-1}\big(x^2\big).\label{eqw4}
 \end{gather}
\end{subequations}
\end{Proposition}

\begin{Theorem}\label{ThmW+}
Let $a$, $b$, $c$ and $d$ be such that $\operatorname{Re}(a,b,c,d)>0$. Consider $k_1,k_2,k_3,k_4\in\{0,1,\dots,$ $n-1\}$, such that $k_1+k_2+k_3+k_4\le n-1$. The sequence $\big\{\tilde{W}_n\big(x^2;a-k_1,b-k_2,c-k_3,d-k_4\big)\big\}_{n=0}^{\infty}$, with $0<\operatorname{Re}(a)<1$ $($if $k_1\ne 0)$, $0<\operatorname{Re}(b)<1$ $($if $k_2\ne 0)$, $0<\operatorname{Re}(c)<1$ $($if $k_3\ne 0)$ and $0<\operatorname{Re}(d)<1$ $($if $k_4\ne 0)$, is quasi-orthogonal of order $k_1+k_2+k_3+k_4\le n-1$ with respect to the weight $w(x)$ on $(0,\infty)$ and the polynomials have at least $n-(k_1+k_2+k_3+k_4)$ real, distinct zeros in $(0,\infty)$.
\end{Theorem}
\begin{proof}Fix $a$ such that $0<\operatorname{Re}(a)<1$. From Lemma \ref{Brez} and \eqref{eqw1}, it follows that $\tilde{W}_n(x^2;a-1$, $b,c,d)$ is quasi-orthogonal of order one on $(0,\infty)$. By iteration, we can express $\tilde{W}_n\big(x^2;a-k$, $b,c,d\big)$ as a linear combination of $\tilde{W}_n\big(x^2;a,b,c,d\big),\tilde{W}_{n-1}\big(x^2;a,b,c,d\big),\dots,\tilde{W}_{n-k}\big(x^2;a,b,c,d\big)$ and from Lemma~\ref{Brez} it follows that $\tilde{W}_n\big(x^2;a-k,b,c,d\big)$, $0<\operatorname{Re}(a)<1$, is quasi-orthogonal of order $k\le n-1$ on $(0,\infty)$. Furthermore, from Lemma~\ref{location} we know that at least~$(n-k)$ real, distinct zeros of $\tilde{W}_n\big(x^2;a-k,b,c,d\big)$, $k\in\{1,2,\dots,n-1\}$, lie in $(0,\infty)$, i.e., at least $(n-k)$ of the zeros $(x_{n,1})^2, (x_{n,2})^2,\dots,(x_{n,n})^2$, lie in~$(0,\infty)$.

When we fix the parameter $b$, (or $c$, $d$) such that $0<\operatorname{Re}(b)<1$ (or $0<\operatorname{Re}(c)<1$, $0<\operatorname{Re}(d)<1$), we can prove in the same way, using \eqref{eqw2} (or \eqref{eqw3}, \eqref{eqw4}),
 that the polynomial $\tilde{W}_n\big(x^2;a,b-k,c,d\big)$ (alternatively $\tilde{W}_n\big(x^2;a,b,c-k,d\big)$, or $\tilde{W}_n\big(x^2;a,b,c,d-k\big)$) is quasi-orthogonal of order~$k$ on $(0,\infty)$. Using an iteration process, we can write $\tilde{W}_n\big(x^2;a-k_1$, $b-k_2,c-k_3,d-k_4\big)$ with $0<\operatorname{Re}(a)<1$ (if $k_1\ne 0$), $0<\operatorname{Re}(b)<1$ (if $k_2\ne 0$), $0<\operatorname{Re}(c)<1$ (if $k_3\ne 0$) and $0<\operatorname{Re}(d)<1$ (if $k_4\ne 0$), in the form of~\eqref{charac_quasi} and the results follow from Lemmas~\ref{Brez} and~\ref{location}.
\end{proof}

\begin{Theorem}\label{Thmawlocation+}Consider $a$, $b$, $c$, $d$, such that $\operatorname{Re}(b,c,d)>0$, $0<\operatorname{Re}(a)<1$ and non-real parameters occur in conjugate pairs. Let $x_{n,i}^2$, $i\in\{1,2,\dots,n\}$, denote the zeros of $\tilde{W}_n\big(x^2;a,b,c,d\big)$ and $y_{n,i}^2$, $i\in\{1,2,\dots,n\}$, the zeros of $\tilde{W}_n\big(x^2;a-1,b,c,d\big)$. Then
\begin{gather*} y_{n,1}^2<x_{n,1}^2<x_{n-1,1}^2<y_{n,2}^2<x_{n,2}^2<\cdots<x_{n-1,n-1}^2<y_{n,n}^2<x_{n,n}^2.\end{gather*}
\end{Theorem}
\begin{proof}From \eqref{eqw1}, we obtain
\begin{gather*} a_n={\frac {n( c+d+n-1)( b+d+n-1)( b+c+n-1) }{( 2n+a+b+c+d-2)( 2n+a+b+c+d-3) }},\end{gather*} which is positive and the interlacing result, as well as the position of $y_{n,n}^2$, follows from Lem\-ma~\ref{Joulak2}(ii).
\end{proof}

\subsection{The Racah polynomials}
 The Racah polynomials
\begin{gather*}\tilde{R}_n(\lambda(x);\alpha,\beta,\gamma,\delta)=\frac{(\alpha+1, \beta+\delta+1, \gamma+1)_n}{(n+\alpha+\beta+1)_n}\\
\hphantom{\tilde{R}_n(\lambda(x);\alpha,\beta,\gamma,\delta)=}{}
\times \hypergeom{4}{3}{-n,n+\alpha+\beta+1, -x,x+\gamma+\delta+1}{\alpha+1, \beta+\delta+1, \gamma+1}{1}, \end{gather*}
$n\in\{0,1,2,\ldots,N\}$, with $\lambda(x)=x(x+\gamma+\delta+1)$, are orthogonal on~$(0,N)$ with respect to the weight function
\begin{gather*}
w(x)=\frac{(\a+1)_x (\b+\delta+1)_x (\g+1)_x (\g+\delta+1)_x((\g+\delta+3)/2)_x}{(-\a+\g+\delta+1)_x(-b+\g+1)_x(\delta+1)_x ((\g+\delta+1)/2)_x}
\end{gather*}
if $\alpha+1=-N$ or $\beta+\delta+1=-N$ or $\gamma+1=-N$ with $N$ a nonnegative integer. Since shifting $\g$ or $\delta$ will change $\lambda(x)$, we will only consider shifts in $\a$ and $\b$.

\begin{Proposition}
\begin{subequations}
 \begin{gather}
\tilde{R}_n(\lambda(x);\alpha-1,\beta,\gamma,\delta) =\tilde{R}_n(\lambda(x);\alpha,\beta,\gamma,\delta) \nonumber\\
\hphantom{\tilde{R}_n(\lambda(x);\alpha-1,\beta,\gamma,\delta) =}{} -{\frac {
 \left( \beta+n \right) \left( \beta+\delta+n \right) \left( \gamma+
n \right) n }{\left( 2 n+\alpha+\beta
 \right) \left( 2 n+\alpha+\beta-1 \right) }}\tilde{R}_{n-1}(\lambda(x);\alpha,\beta,\gamma,\delta)\label{eqr1};\\
\tilde{R}_n(\lambda(x);\alpha,\beta-1,\gamma,\delta) =\tilde{R}_n(\lambda(x);\alpha,\beta,\gamma,\delta)\nonumber\\
\hphantom{\tilde{R}_n(\lambda(x);\alpha,\beta-1,\gamma,\delta) =}{} -{\frac { \left(
\alpha+n \right) \left( \alpha-\delta+n \right) \left( \gamma+n
 \right) n }{ \left( 2 n+\alpha+\beta
 \right) \left( 2 n+\alpha+\beta-1 \right) }}\tilde{R}_{n-1}(\lambda(x);\alpha,\beta,\gamma,\delta).\label{eqr2}
 \end{gather}
\end{subequations}
\end{Proposition}

\begin{Theorem}\label{ThmW++}
 Let $k\in\{1,2,\dots,n-1\}$. The sequence of Racah polynomials
\begin{itemize}\itemsep=0pt
\item[$(i)$] $\big\{\tilde{R}_n(\lambda(x);\alpha-k,\beta,\gamma,\delta)\big\}_{n=0}^{N}$, with $\a=-N-1$, is quasi-orthogonal of order $k$ with respect to the weight $w(x)$ on $(0,\lambda(N))$ and the polynomials have at least $(n-k)$ real, distinct zeros in $(0,\lambda(N))$;
\item[$(ii)$] $\big\{\tilde{R}_n(\lambda(x);\alpha,\beta-k,\gamma,\delta)\big\}_{n=0}^{N}$, with $\b=-N-\delta-1$, is quasi-orthogonal of order $k$ with respect to the weight $w(x)$ on $(0,\lambda(N))$ and the polynomials have at least $(n-k)$ real, distinct zeros in $(0,\lambda(N))$.
\end{itemize}
\end{Theorem}
\begin{proof} (i) Let $\a=-N-1$. From Lemma~\ref{Brez} and~\eqref{eqr1}, it follows that $\tilde{R}_n(\lambda(x);\alpha-1,\beta,\gamma,\delta)$ is quasi-orthogonal of order one on $(0,\lambda(N))$. By iteration, we can express $\tilde{R}_n(\lambda(x);\alpha-k,\beta,\gamma,\delta)$ as a linear combination of $\tilde{R}_n(\lambda(x);\alpha,\beta,\gamma,\delta)$, $\tilde{R}_{n-1}(\lambda(x);\alpha,\beta,\gamma,\delta)$, $\dots$, $\tilde{R}_{n-k}(\lambda(x);\alpha,\beta,\gamma,\delta)$ and the result follows from Lemma~\ref{Brez}. Furthermore, from Lemma~\ref{location} we know that at least $(n-k)$ real, distinct zeros of $\tilde{R}_n(\lambda(x);\alpha-k,\beta,\gamma,\delta)$, $k\in\{1,2,\dots,n-1\}$, lie in $(0,\lambda(N))$.

(ii) Let $\b=-N-\delta-1$. The result follow in the same way from~\eqref{eqr2} and Lemmas~\ref{Brez} and~\ref{location}.
\end{proof}

As in the case of the $q$-Racah polynomials, we obtain different interlacing results for values of $n$ larger than $\frac{N}{2}+1$, that we show in the next theorem, than for $n<\frac{N}{2}+1$.
\begin{Theorem} Consider $n\le N$ and let $x_{n,i}$, $i\in\{1,2,\dots,n\}$ denote the zeros of $\tilde{R}_n(\lambda(x);\alpha,\beta$, $\gamma,\delta)$, $y_{n,i}$, $i\in\{1,2,\dots,n\}$, the zeros of $\tilde{R}_n(\lambda(x);\alpha-1,\beta,\gamma,\delta)$ and $z_{n,i}$, $i\in\{1,2,\dots,n\}$, the zeros of
$\tilde{R}_n(\lambda(x);\alpha,\beta-1,\gamma,\delta)$. Then, for $n>\frac{N}{2}+1$,
\begin{itemize}\itemsep=0pt
\item[$(i)$] if $\alpha=-N-1$ and $\beta>0$, $\delta>0$, $\gamma>0$, we have \begin{gather*} 0<x_{n,1}<y_{n,1}<x_{n-1,1}<x_{n,2}<y_{n,2}<\cdots<x_{n-1,n-1}<x_{n,n}<y_{n,n};\end{gather*}
\item[$(ii)$] if $\beta =-N-\delta-1$ and $\alpha>0$, $\gamma>0$, $\alpha-\delta>0$, we have
\begin{gather*} 0<x_{n,1}<z_{n,1}<x_{n-1,1}<x_{n,2}<z_{n,2}<\cdots<x_{n-1,n-1}<x_{n,n}<z_{n,n}.\end{gather*}
\end{itemize}
\end{Theorem}
\begin{proof} Under the above hypotheses,
 the coefficients of $\tilde{R}_n(\lambda(x);\alpha,\beta,\gamma,\delta)$ in \eqref{eqr1} and \eqref{eqr2} are negative and the interlacing results follow from Lemma \ref{Joulak2}(i).
 \end{proof}
\subsection{The continuous Hahn polynomials}
The continuous Hahn polynomials
\begin{gather*}\tilde{P}_n(x;a,b,c,d)=\frac{i^n(a+c,a+d)_n}{(n+a+b+c+d-1)_n}\, \hypergeom{3}{2}{-n,n+a+b+c+d-1,a+ix}{a+c,a+d}{1}\end{gather*}
are orthogonal on $\rr$ with respect to
\begin{gather*}
 w(x)=\Gamma(a+ix)\Gamma(b+ix)\Gamma(c-ix)\Gamma(d-ix)
 \end{gather*} for $\operatorname{Re}(a,b,c,d)>0$, $c=\bar{a}$ and $d=\bar{b}$.
\begin{Proposition}
\begin{subequations}
\begin{gather}
\tilde{P}_n(x;a-1,b,c,d) =\tilde{P}_n(x;a,b,c,d) \nonumber\\
 \qquad{} +{\frac {i \left( b+c+n-1 \right) \left( b+d+n-1 \right) n}{ \left( 2
 n+a+b+c+d-3 \right) \left( 2n+a+b+c+d-2 \right) }}\tilde{P}_{n-1}(x;a,b,c,d);\label{eqch1}\\
\tilde{P}_{n}(x;a,b-1,c,d) =\tilde{P}_{n}(x;a,b,c,d)\nonumber\\
\qquad{} +{\frac {i \left( a-1+d+ n \right) \left( a-1+c+n \right) n }{ \left( 2
n+a+b+c+d-3 \right) \left( 2n+a+b+c+d-2 \right) }}\tilde{P}_{n-1}(x;a,b,c,d);\label{eqch2}\\
\tilde{P}_{n}(x;a,b,c-1,d) =\tilde{P}_{n}(x;a,b,c,d)\nonumber\\
 \qquad{} -{\frac {i \left( b+d+n-1 \right) \left( a-1+d+n \right) n }{ \left( 2
n+a-3+b+c+d \right) \left( 2n+a-2+b+c+d \right) }}\tilde{P}_{n-1}(x;a,b,c,d);\label{eqch3}\\
\tilde{P}_{n}(x;a,b,c,d-1) =\tilde{P}_{n}(x;a,b,c,d)\nonumber\\
\qquad{} -{\frac {i \left( b+c+n-1 \right) \left( a-1+c+n \right) n }{ \left( 2
n+a-3+b+c+d \right) \left( 2n+a-2+b+c+d \right) }}\tilde{P}_{n-1}(x;a,b,c,d).\label{eqch4}
 \end{gather}
\end{subequations}
\end{Proposition}

\begin{Corollary}
\begin{subequations}
 \begin{gather}
 \tilde{P}_n(x;a-1,b,c-1,d) \nonumber\\
\quad {} =\tilde{P}_n(x)-{\frac {i \left( a+d-b-c \right) \left( b+d+n-1 \right) n}{ \left( 2n+a-4+b+c+d \right) \left( 2n+a-2+b+c+d
 \right) }}\tilde{P}_{n-1}(x)\label{eqch5}\\
 \quad {} +{\frac { \left( b+d+n-2 \right) \left( a-2+d+n \right)
 \left( n-1 \right) \left( b+c-2+n
 \right) \left( b+d+n-1 \right) n}{ \left( 2n-5+a+b+c+d \right)
 \left( 2n+a-4+b+c+d \right) ^{2} \left( 2n+a-3+b+c+d \right) }}\tilde{P}_{n-2}(x);\nonumber\\
 \tilde{P}_{n}(x;a,b-1,c,d-1) \nonumber\\
 \quad{} {}=\tilde{P}_{n}(x) +{\frac {i
 \left( a+d-b-c \right) \left( a-1+c+n \right) n }{ \left( 2n-4+a+b+c+d \right) \left( 2n+a-2+b+c+d
 \right) }}\tilde{P}_{n-1}(x)\label{eqch6}\\
\quad {} +{\frac {n \left( a+d+n-2 \right) \left( a-1+c+n \right)
 \left( b+c+n-2 \right) \left( a+c+n-2 \right) \left( n-1 \right) }{ \left( 2n-5+a+b+c+d \right) \left( 2
n-4+a+b+c+d \right) ^{2} \left( 2n+a-3+b+c+d \right) }}\tilde{P}_{n-2}(x).\nonumber
 \end{gather}
\end{subequations}
\end{Corollary}

\begin{Theorem}\label{ThmW+++}
Consider $a$, $b$, $c$, $d$ such that $\operatorname{Re}(a,b,c,d)>0$, $c=\bar{a}$ and $d=\bar{b}$. Consider $k_1,k_2,k_3,k_4\in\{0,1,\dots,n-1\}$, such that $k_1+k_2+k_3+k_4\le n-1$. The sequence of continuous Hahn polynomials $\{\tilde{P}_n(x;a-k_1,b-k_2,c-k_3,d-k_4)\}_{n=0}^{\infty}$, with $0<\operatorname{Re}(a)=\operatorname{Re}(c)<1$ $($if $k_1\ne 0)$, $0<\operatorname{Re}(b)=\operatorname{Re}(d)<1$ $($if $k_2\ne 0)$, $0<\operatorname{Re}(a)=\operatorname{Re}(c)<1$ $($if $k_3\ne 0)$ and $0<\operatorname{Re}(b)=\operatorname{Re}(d)<1$ $($if $k_4\ne 0)$, is quasi-orthogonal of order $k_1+k_2+k_3+k_4\le n-1$ with respect to the weight $w(x)$ on~$\rr$ and the polynomials have at least $n-(k_1+k_2+k_3+k_4)$ real, distinct zeros.
\end{Theorem}
\begin{proof}

Fix $a$ and $c$ such that $0<\operatorname{Re}(a)=\operatorname{Re}(c)<1$. From Lemma~\ref{Brez} and~\eqref{eqch1}, it follows that $\tilde{P}_n(x;a-1,b,c,d)$ is quasi-orthogonal of order one on $\rr$. By iteration, we can express $\tilde{P}_n(x;a-k,b,c,d)$ as a linear combination of $\tilde{P}_n(x;a,b,c,d),\tilde{P}_{n-1}(x;a,b,c,d),\dots,\tilde{P}_{n-k}(x;a,b$, $c,d)$ and it follows from Lemma~\ref{Brez} that $\tilde{P}_n(x;a-1,b,c,d)$ is quasi-orthogonal of order one on~$\rr$. By using an iteration process, we can write $\tilde{P}_n(x;a-k,b,c,d)$ as a linear combination of ortho\-go\-nal continuous Hahn polynomials and it is quasi-orthogonal of order $k\le n-1$. Furthermore, from Lemma~\ref{location} we know that at least~$(n-k)$ zeros of $\tilde{P}_n(x;a-k,b,c,d)$, $k\in\{1,2,\dots,n-1\}$, are real and distinct. In the same way, using~\eqref{eqch3}, we can prove that $\tilde{P}_n(x;a,b,c-k,d)$, $0<\operatorname{Re}(a)=\operatorname{Re}(c)<1$ is quasi-orthogonal of $k\le n-1$ on~$\rr$. By fixing~$b$ and~$d$ such that $0<\operatorname{Re}(b)=\operatorname{Re}(d)<1$, we can prove the quasi-orthogonality of $\tilde{P}_n(x;a,b-k,c,d)$ and $\tilde{P}_n(x;a,b,c,d-k)$, using \eqref{eqch2}, \eqref{eqch4} and Lemma~\ref{Brez}.

 Using an iteration process, we can write $\tilde{P}_n(x;a-k_1,b-k_2,c-k_3,d-k_4)$, with $0<\operatorname{Re}(a)=\operatorname{Re}(c)<1$ (if $k_1\ne 0$), $0<\operatorname{Re}(b)=\operatorname{Re}(d)<1$ (if $k_2\ne 0$), $0<\operatorname{Re}(a)=\operatorname{Re}(c)<1$ (if $k_3\ne 0$) and $0<\operatorname{Re}(b)=\operatorname{Re}(d)<1$ (if $k_4\ne 0$), as a linear combination of orthogonal continuos Hahn polynomials and the results follow from Lemmas~\ref{Brez} and~\ref{location}.
\end{proof}

\begin{Theorem}\label{ThmW++++}
Consider $a,b,c,d$ such that $\operatorname{Re}(a,b,c,d)>0$, $c=\bar{a}$ and $d=\bar{b}$.
\begin{itemize}\itemsep=0pt
\item[$(i)$] Let $0<\operatorname{Re}(a)= \operatorname{Re}(c)<1$. Then $(n-2)$ zeros of $\tilde{P}_n(x;a-1,b,c-1,d)$ interlace with the zeros of $\tilde{P}_{n-1}(x;a,b,c,d)$;
\item[$(ii)$] Let $0<\operatorname{Re}(b)= \operatorname{Re}(d)<1$. Then $(n-2)$ zeros of $\tilde{P}_n(x;a,b-1,c,d-1)$ interlace with the zeros of $\tilde{P}_{n-1}(x;a,b,c,d)$.
\end{itemize}
\end{Theorem}
\begin{proof}In this proof $-C_n$ refers to the coefficient of $\tilde{P}_{n-2}(x;a,b,c,d)$ in the three-term recurrence equation of the continuous Hahn polynomials (cf.~\cite[equation~(9.4.3)]{KLS}), involving the polynomials $\tilde{P}_{n}(x;a,b,c,d)$, $\tilde{P}_{n-1}(x;a,b,c,d)$ and $\tilde{P}_{n-2}(x;a,b,c,d)$.

(i) Let $0<\operatorname{Re}(a), \operatorname{Re}(c)<1$. We consider the coefficient $b_n$ of $\tilde{P}_{n-2}(x;a,b,c,d)$ in \eqref{eqch5}. Then
\begin{gather*} C_n-b_n={\frac { \left( a+c-2 \right) \left( n-2+b+d \right) \left( a+d+n-2
 \right) \left( n-1 \right) \left( b+c+n-2 \right) }{ \left( 2n-5+
a+b+c+d \right) \left( 2n-4+a+b+c+d \right) ^{2}}},\end{gather*} and when we take into consideration the specific restrictions on the parameters, we observe that $C_n<b_n$ and the result follows from \cite[Theorem~15(ii)]{Joulak_2005}.

(ii) Let $0<\operatorname{Re}(b), \operatorname{Re}(d)<1$ and let $b_n$ be the coefficient of $\tilde{P}_{n-2}(x;a,b,c,d)$ in \eqref{eqch6}. Then
\begin{gather*} C_n-b_n={\frac { \left( b+d-2 \right) \left( n-1 \right) \left( b+c+n-2
 \right) \left( a+d+n-2 \right) \left( a+c+n-2 \right) }{ \left( 2 n-5+a+b+c+d \right) \left( 2n-4+a+b+c+d \right) ^{2}}}<0,\end{gather*} when we take into consideration the specific restrictions on the parameters and the result follows from \cite[Theorem~15(ii)]{Joulak_2005}.
\end{proof}

\section{Concluding remarks}\label{section5}
The $q$-Meixner polynomials, defined by
\begin{gather*}\tilde{M}_n(\bar{x};\beta,\gamma;q)=(-1)^nq^{-n^2}\gamma^n(\beta q;q)_n\,\qhypergeom{2}{1}{q^{-n},\bar{x}}{\beta q}{q;-{q^{n+1}\over \gamma}},\end{gather*}
 with $\bar{x}=q^{-x}$, are orthogonal with respect to the discrete weight $\frac{(\b q;q)_x \g^x q^{\binom{x}{2}}}{(q,-\b \g q;q)_x}$, when $0\leq \beta q<1$, $\gamma>0$, $\bar{x}\in(1,\infty)$, and satisfy
\begin{gather*} \tilde{M}_n\left(\bar{x};\frac{\beta}{q},\gamma;q\right) ={\frac { \left( {q}^{n}x+\beta
 \gamma \right) }{(\beta \gamma+x){q}^{n}}}\tilde{M}_n(\bar{x};\beta,\gamma;q)-
{\frac { \b \g \left( {q}^{n}+\gamma \right) \left( {q}^{n}-1 \right)q}{(\beta \gamma+x){q}^{3 n}}}\tilde{M}_n(\bar{x};\beta,\gamma;q).
\end{gather*}
The polynomial $\tilde{M}_n\big(\bar{x};\frac{\beta}{q^k},\gamma;q\big)$, $k<n$, is not quasi-orthogonal with respect to $\frac{(\b q;q)_x \g^x q^{\binom{x}{2}}}{(q,-\b \g q;q)_x}$, on~$(1,\infty)$, since it cannot be written as a linear combination of the polynomials $\tilde{M}_n(\bar{x};\beta,\gamma;q)$, $\tilde{M}_{n-1}(\bar{x};\beta,\gamma;q), \dots,\tilde{M}_{n-k}(\bar{x};\beta,\gamma;q)$. Since $\gamma>0$, we also have $\frac{\gamma}{q}>0$ or $\gamma q>0$ and the sequences $\tilde{M}_n\big(\bar{x};\beta,\frac{\gamma}{q};q\big)$ or $\tilde{M}_n(\bar{x};\beta,\gamma q;q)$, are orthogonal on $(1,\infty)$ for $0\leq \beta q<1$. We therefore do not consider $q$-shifts of~$\gamma$.

The same is true for the Al-Salam--Carlitz II polynomials \cite[Section 14.25]{KLS}, that satisfy
\begin{gather*} \tilde{V}_{n}^{(\frac {\alpha}{q})}(x;q) =-{\frac {( {q}^{n}x-\alpha )}{(\alpha-x)q^n}}\tilde{V}_{n}^{(\a)}(x;q)-{\frac {
\a {q} \left( {q}^{n}-1 \right)}{(\alpha-x){q}^{2n}}}\tilde{V}_{n-1}^{(\a)}(x;q)\end{gather*} and the Bessel polynomials \cite[Section~9.13]{KLS}, that satisfy the equation
\begin{gather*}
\tilde{y}_{n}(x;\alpha+1) ={\frac {{\alpha}^{2}x+4\alpha
nx+4{n}^{2}x+\alpha x+2nx-2n
}{x( \alpha+1+2n) ( \alpha+2n ) }}\tilde{y}_{n}(x;\alpha)\\
\hphantom{\tilde{y}_{n}(x;\alpha+1) =}{} +{
\frac {4n( \alpha+n) }{x( \alpha+1+2) ( \alpha+2n-1) (\alpha+2n) ^2}}\tilde{y}_{n-1}(x;\alpha).
\end{gather*}

The $q$-Krawtchouk polynomials
\begin{gather*}\tilde{K}_n(\bar{x};p,N;q)=\frac{(q^{-N};q)_n}{(-pq^n;q)_n}\,\qhypergeom{3}{2}{q^{-n},\bar{x},-pq^n}{q^{-N},0}{q;q},\end{gather*}
 with $ \bar{x}=q^{-x}$ and $n\in\{0,1,\ldots,N\}$, are orthogonal for $p>0$ with respect to the discrete weight $w(x)=\frac{(q^{-N};q)_x(-p)^x}{(q;q)_x}$ on $\big(1,q^{-N}\big)$. The polynomials $\tilde{K}_n\big(\bar{x};\frac{p}{q^k},N;q\big)$ are orthogonal for $p>0$ with respect to $\frac{(q^{-N};q)_x(\frac{-p}{q^k})^x}{(q;q)_x}$ on $\big(1,q^{-N}\big)$. By iterating the equation
\begin{gather*} K_ n\left(\bar{x};\frac{p}{q},N;q\right) =K_ n(\bar{x};p,N;q) -{\frac {p
 \left( {q}^{n}-1 \right) \left( {q}^{n}-{q}^{N+1} \right){q}^{n+1} }{ \left( {q}^{2n}p+q \right) \left( {q}^{2n}p+{q}^{2} \right)q^N }} K_ {n-1}(\bar{x};p,N;q),\end{gather*}
we can write $\tilde{K}_n(\bar{x};\frac{p}{q^k},N;q)$ as a linear combination of the polynomials $\tilde{K}_{n-j}(\bar{x};{p},N;q)$, $j\in\{0,1,\dots,k\}$, and the polynomials $\tilde{K}_n\big(\bar{x};\frac{p}{q^k},N;q\big)$ are also quasi-orthogonal for $p>0$ on $\big(1,q^{-N}\big)$ with respect to $w(x)$.

The same is true for the $q$-Charlier \cite[Section~14.23]{KLS} and alternative $q$-Charlier (or $q$-Bessel) polynomials \cite[Section~14.22]{KLS}.

\subsection*{Acknowledgments}
The authors thank the referees for the valuable comments and suggestions which considerably improved the manuscript. This work has been supported by the Institute of Mathematics of the University of Kassel (Germany) for D.D.~Tcheutia.

\pdfbookmark[1]{References}{ref}
\LastPageEnding

\end{document}